\RequirePackage{fix-cm}
\documentclass[envcountsect,numbook,natbib]{svjour3_arxiv4}
\smartqed  
\usepackage{graphicx}
\graphicspath{{./figures/}}
\usepackage[utf8]{inputenc}
\usepackage{dsfont}
\usepackage{subfigure}
\usepackage{color}
\usepackage{url}
\usepackage{enumerate} 
\usepackage{booktabs}
\usepackage{tabularx}
\usepackage{ltablex}
\usepackage{mathptmx}
\usepackage{amssymb,amsmath,amsfonts,latexsym} 
\usepackage{longtable} 
\usepackage{textcomp}
\usepackage{lscape}
\usepackage[gen]{eurosym}  
\usepackage{bm}
\usepackage{bbm}
\usepackage{multirow}
\usepackage{calrsfs}  
\newcommand{\Q}{\pazocal{Q}}  
\DeclareMathOperator*{\argmin}{arg\,min}
\usepackage[12pt]{extsizes}
\usepackage{geometry}
\usepackage{ifthen}
\allowdisplaybreaks

\spnewtheorem{theorem}{Theorem}{\bfseries}{\itshape}
\spnewtheorem{corollary}[theorem]{Corollary}{\bfseries}{\itshape}
\spnewtheorem{lemma}[theorem]{Lemma}{\bfseries}{\itshape}
\spnewtheorem{proposition}[theorem]{Proposition}{\bfseries}{\itshape}
\spnewtheorem{definition}[theorem]{Definition}{\bfseries}{\itshape}
\spnewtheorem{remark}[theorem]{Remark}{\bfseries}{\upshape}
\spnewtheorem{assumption}[theorem]{Assumption}{\bfseries}{\itshape}
\counterwithin{theorem}{section}
\smartqed

\renewcommand{\paragraph}[1]{{\bf #1.}}
\definecolor{myred}{rgb}{0.8,0,0}  

{\vskip\baselineskip\noindent\textbf{Proof of {#1}:}}%
{\hspace*{.1pt}\hspace*{\fill}\BOX\vskip\baselineskip}

\usepackage{tikz}

\def \Q{\mathbb{Q}}             
\def \1{{\bf 1}}                
\def \0{{\bf 0}}
\def\qed{\hfill$\Box$}

\definecolor{myred}{rgb}{0.9,0,0}  
\definecolor{mygreen}{rgb}{0,0.7,0}  
\definecolor{myblue}{rgb}{0.2,0,0.8}  
\definecolor{orange}{rgb}{1,0.6,0}  
\definecolor{olive}{rgb}{0.5,0.5,0}  
\definecolor{mylila}{rgb}{0.8,0.5,0.2}  
\definecolor{mygrey}{rgb}{0.6,0.6,0.6}  
\definecolor{mybrown}{rgb}{0.65,0.16,0.16}  
\definecolor{mymaroon}{rgb}{0.11,0.0,0.0}

\newcommand{\ZR}{Z}
\newcommand{\zR}{z}
\newcommand{\muR}{\mu}
\newcommand{\WR}{W}
\newcommand{\kappaR}{\kappa}
\newcommand{\sigmaR}{\sigma}
\newcommand{\tildemu}{\tilde{\mu}}
\DeclareMathAlphabet{\pazocal}{OMS}{zplm}{m}{n}
\let\mathcal\pazocal
\begin{document}
	
	\title{Stochastic Optimal Control of Prosumers in a District Heating System.  \\  
		\thanks{This work was carried out with the aid of a grant from the International Development Research Centre, Ottawa, Canada, www.idrc.ca, and with financial support from the Government of Canada, provided through Global Affairs Canada (GAC), www.international.gc.ca;
			Brandenburg University of Technology Cottbus-Senftenberg, Erasmus+ and research scholarship.\newline
			The Author gratefully acknowledges the  support by the German Academic Exchange Service, Grant/Award Number: 57417894.
		}
}

\titlerunning{Stochastic Optimal Control of Prosumers in a District Heating System.}        

\author{Maalvlad\'edon Ganet Som\'e}

\authorrunning{M. Ganet Som\'e, R. Wunderlich} 

\institute{Maalvlad\'edon Ganet Som\'e \at
	University of Rwanda, Department of Mathematics, P.O. Box 4285, Kigali , Kigali, Gikondo-Street, KK 737; \\
	African Institute for Mathematical Sciences, Ghana, P.O. Box LGDTD 20046, Legon, Accra, ! Shoppers Street, Spintex, Accra, Ghana;  
	\email{\texttt{maalvladedon@aims.edu.gh}}           
}

\date{Version of  \today}

\maketitle

\begin{abstract}
	We consider a network of residential heating systems in which several prosumers satisfy their heating and hot
	water demand using solar thermal collectors and services of a central producer. Overproduction of heat can
	either be stored in a local thermal storage or sold to the network. Our focus is the minimization of the prosumers'
	expected discounted total cost from purchasing and selling thermal energy and running the system. This decision making
	problem under uncertainty about the future production and consumption of thermal energy is formulated as a
	stochastic optimal control problem and solved with dynamic programming techniques. We present numerical
	results for the value function and the optimal control.
	
	%
\end{abstract}	

\section{Introduction}
The use of fossil fuels is gradually becoming problematic given its limited nature and the threats (including pollution and environmental problems) their use pose to our planet. In recent years, most researchers
\cite{aid2020mckean,alasseur2020extended,alasseur2023mfg} focused on the electricity system by promoting the penetration of renewable energy sources as a way to address energy transition and to meet the Paris agreement which is to keep the rise in mean global temperature preferably to $1.5^{\circ}C$.  However studies \cite{alam2008new,ghisi2007electricity,hepbasli2009review,liu2014optimization,perez2008review,singh2010factors} reveal that in the US and parts of Europe, for buildings, heating and hot water account for $18\%$ to $30\%$ of total energy consumption. Hence this requires urgent measures to improve the energy efficiency in buildings and all areas in order to mitigate the impact of climate change. Thermal networks like district heating systems are considered to be essential to boost energy system efficiency and save cost \cite{connolly2012heat,lund2012heat}.\newline 
District heating is an efficient means of supplying heat to buildings. It comprises a network of pipes which connect buildings in a neighbourhood or city for a reliable heat delivery. 
We consider a residential heating system or household (see Figure \ref{Prosumermodel}) comprising a local renewable heat production source (solar collector) and storage units taken as water tanks for simplicity and equipped with heat and ordinary pumps. Heat production occurs through a solar collector. In our model, the internal storage (IS) serves as a short term storage unit while the external storage (ES) is used for a long term storage. Since the solar collector's production is intermittent and highly dependent on weather conditions, we connect our household to a community heating system (CHS) for a reliable heat supply at a given cost. The CHS we have in mind is a cold district heating network which is a variant of traditional district heating networks and operates at low transission temperatures. These low transmission temperatures require the usage of a heat pump. For further reference on district heating systems and low energy emission buildings, interested readers may refer to \cite{chen2013energy,heiselberg2009application,nielsen2012excess,tommerup2007energy,tommerup2006energy} and their refernces therein.\newline
We do not model  all details of heat production and transmission in the building. In lieu of that, we motivate as in \cite{alasseur2020extended} a residual demand, which in our modelling perspective is the difference between consumption (demand) and production (supply).  
The manager of the residential heating system is referred to as a prosumer since it is both a producer and consumer of heat. We further assume that the CHS has the obligation to satisfy the residual demand of the prosumer. 
Throughout this project, we assume that the residual demand can either be satisfied using the services of the CHS or the ES.\newline
For an unsatisfied demand, the prosumer can either purchase heat from CHS at a cost or escape CHS market power by discharging ES provided it is not empty. Becasue of the fact that we consider a cold district heating, purchasing heat from CHS requires the use of a heat pump which generates additional cost due to high electricity consumption while the discharging ES requires the use of an ordinary pump which rather generates a relatively low cost. Similarly, for an overproduction, the prosumer can either sell excess heat to CHS for a revenue, or charge ES subject to loss to the environment and pay a cost from operating the ordinary pump.\newline
The goal of the prosumer is to minimize the expected total aggregated discount cost of satisfying its hot water and heating demand. This leads to a challenging optimization problem with uncertainty due to the highly intermittent production of the solar collector and the random consumption of the prosumer affected by the weather conditions.\newline
To the best of our knowledge, our work is closest to \cite{takam2023stochastic}, where the author formulated and solved a stochastic optimal control problem of residential heating systems with a geothermal storage. The work assumes several heat production sources including electricity, fuel and renewable (solar) energy. However in \cite{takam2023stochastic}, the system is standalone in the sense that it is not connected to a CHS. Instead of a geothermal storage, we consider a water tank and connect our system to a CHS where prosumers can satisfy their demand by purchasing or selling heat.\newline
Our contributions include a mathematical formulation of a residential heating system with a careful calibration of model parameters. We also formulate a stochastic optimal control problem with state-dependent control constraints and derived the associated Hamilton-Jacobi-Bellmann (HJB) equation. In addition, based on the dynamic programming principle (DPP), we derive a discrete-time scheme which we solve to obtain the optimal strategies and the value function. We also perform extensive numerical experiments to study the impact of both seasonality and storage insulation to our system.\newline
This paper is structured as follows. In Sect. \ref{Model}, we present our model and its hypotheses, introduce the heat price formulation and discuss the state and control constraints. In Sect. \ref{soc}, we formulate the stochastic optimal control problem, derive a numerical scheme for solving the HJB equation and motivate the boundary conditions. In Sect. \ref{numericalresults}, we discuss our numerical results for various scenarios and conclude in Sect. \ref{conclusion}.

\begin{figure}[ht!]			
	\centering
	\includegraphics[width=1.0\textwidth]{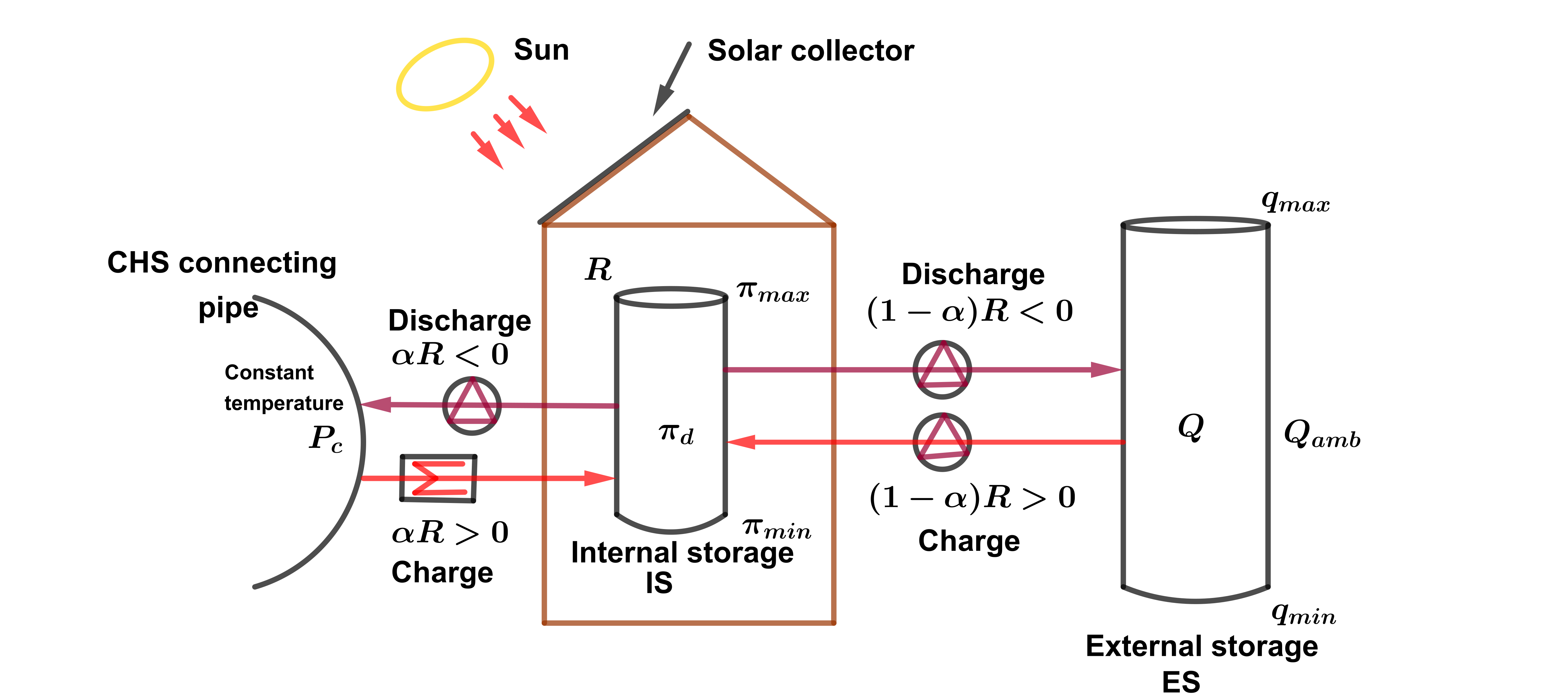}
	\caption[A Prosumer model.]{Model of the prosumer's residential heating system.}
	\label{Prosumermodel}
\end{figure}

\section{Model}
\label{Model}
We formulate the problem of a cost-optimal management of the prosumer's residential heating system in a district heating system as a stochastic optimal control problem in continuous time.\newline
Let $T>0$ be a finite planning horizon and assume a filtered probability space $(\Omega,\mathcal{F},\mathbb{F},\mathbb{P})$ which supports the Brownian motion $\WR$. The filtration $\mathbb{F}=\mathbb{F}^{\WR}=(\mathcal{F}^{\WR}_t)_{t\in[0,T]}$ satisfies the usual conditions. All processes are assumed to be $\mathbb{F}-$adapted.

\subsection{Control system}
We model a control system for the residential heating system of a prosumer which covers the case of a residual demand with seasonality as in \cite{takam2023stochastic}. In our modelling perspective, we assume that the prosumer can purchase or sell thermal energy from or to CHS. We also assume that the interaction between prosumer and CHS when purchasing thermal energy involves a heat pump, while selling to CHS rather requires the service of an ordinary pump. The heat pump increases the temperature from the constant temperature $P_c$ in the connecting pipe to a desidered temperature $\pi_d$ in the IS. In that case, $P_c$ and $\pi_d$ are referred as the heat pump's inlet and outlet temperatures, respectively. The ordinary pump, however serves to move the thermal energy to the connecting pipe.\newline

\subsection{Deseasonalized residual demand}
The dynamics of the deseasonalized residual demand $\ZR\in\mathbb{R}$ is described by a zero-mean reversion level Ornstein-Uhlenbeck (O-U) process of the form
\begin{equation}
	\label{residualdemand}
	\mathrm{d}\ZR(t) = -\kappaR \ZR(t)\mathrm{d}t + \sigmaR\mathrm{d}\WR(t), \quad \ZR(0)=\ZR_0\in \mathbb{R},
\end{equation} 
where the standard Brownian motion $W$ models the uncertainties in solar collector's production and prosumer's heat consumption due to the randomness in the weather factors. \newline
Let the bounded and deterministic function $\muR: [0,T]\rightarrow \mathbb{R}$ denote the seasonality function, modelling the seasonal behaviour of the residual demand with a typical choice as
\begin{equation}
	\label{seasfunc}
	\muR(t)=L_{0}+\sum^{m}_{i=1} L_i\cos\Big(\frac{2\pi}{\rho_i}(t-t_i)\Big), 
\end{equation}
with $L_{0}$ the long term mean residual demand, $L_i$, $\rho_i$ and $t_i$ represent respectively, the amplitude, length and a reference time of the seasonality component $i$, and $m$ is the number of seasonality components. The calibration of $L_0$ and $L_i$ is provided in Appendix \ref{calibrationL_0L}. \newline
A typical choice is $m = 2$ and $\rho_1 = 1$ year, $\rho_2 = $ half-day. Hence $\rho_1$ accounts for the yearly seasonalities and $\rho_2$ the half-day seasonalities due to the peaks in the morning and evening.\newline
Models with time dependent mean reversion level are also proposed in \cite{chen2008semi,takam2023stochastic}.\newline
In the following, we introduce the function $R(t)$ to denote the prosumer's residual demand which we model for all $t\in [0,T]$ as
\begin{equation}
	\label{residdemand}
	R(t) = \muR(t) + \ZR(t).
\end{equation}
We use the convention that $R>0$ (i.e. $\muR+\ZR>0$) refers to unsatisfied demand and $R<0$ (i.e. $\muR+\ZR<0$) to overproduction.\newline
Let $\alpha \in [0,1]$ denote the prosumer's choice of the proportion of residual demand to be satisfied from CHS. Hence, we have the decomposition $R = \alpha R + (1-\alpha)R$ which implies that the prosumer can satisfy its residual demand with $\alpha R$ from CHS and $(1-\alpha)R$ from the ES. In our modelling perspectives, satisfying the residual demand from the CHS means either purchasing (for an unsatisfied demand) or selling (for an overproduction). Similarly, satisfying the residual demand from the ES means either discharging (for an unsatisfied demand) or storing (for an overproduction). In the lemma below, we deduce the differential form of the residual demand $R$ for all $t\in [0,T]$.

\begin{lemma} 
	Let $\muR: [0,T] \rightarrow \mathbb{R}$ be a bounded deterministic differenctiable function. Then for $\tildemu(t) = \muR(t) + \frac{1}{\kappaR}\dot{\mu}(t)$, the residual demand $R$ satisfies the SDE
	\begin{equation}
		\mathrm{d}R(t) = \kappaR(\tildemu(t)-R(t))\mathrm{d}t + \sigma \mathrm{d}\WR(t), \quad R(0) = R_0,
	\end{equation}
\end{lemma}
\begin{proof}
	Recall that $R(t) = \muR(t)+\ZR(t)$. Taking the derivative on both sides, we obtain
	\begin{equation}
		\mathrm{d}R(t) = \dot{\mu}(t)\mathrm{d}t + \mathrm{d}\ZR(t).
	\end{equation}
	The result follows recalling Equation \eqref{residualdemand} and noting that $\dot{\mu}(t) = \kappaR(\tilde{\mu}(t)-\muR(t))$.
\end{proof}

\begin{figure}[ht!]			
	\centering
	\includegraphics[width=1.0\textwidth]{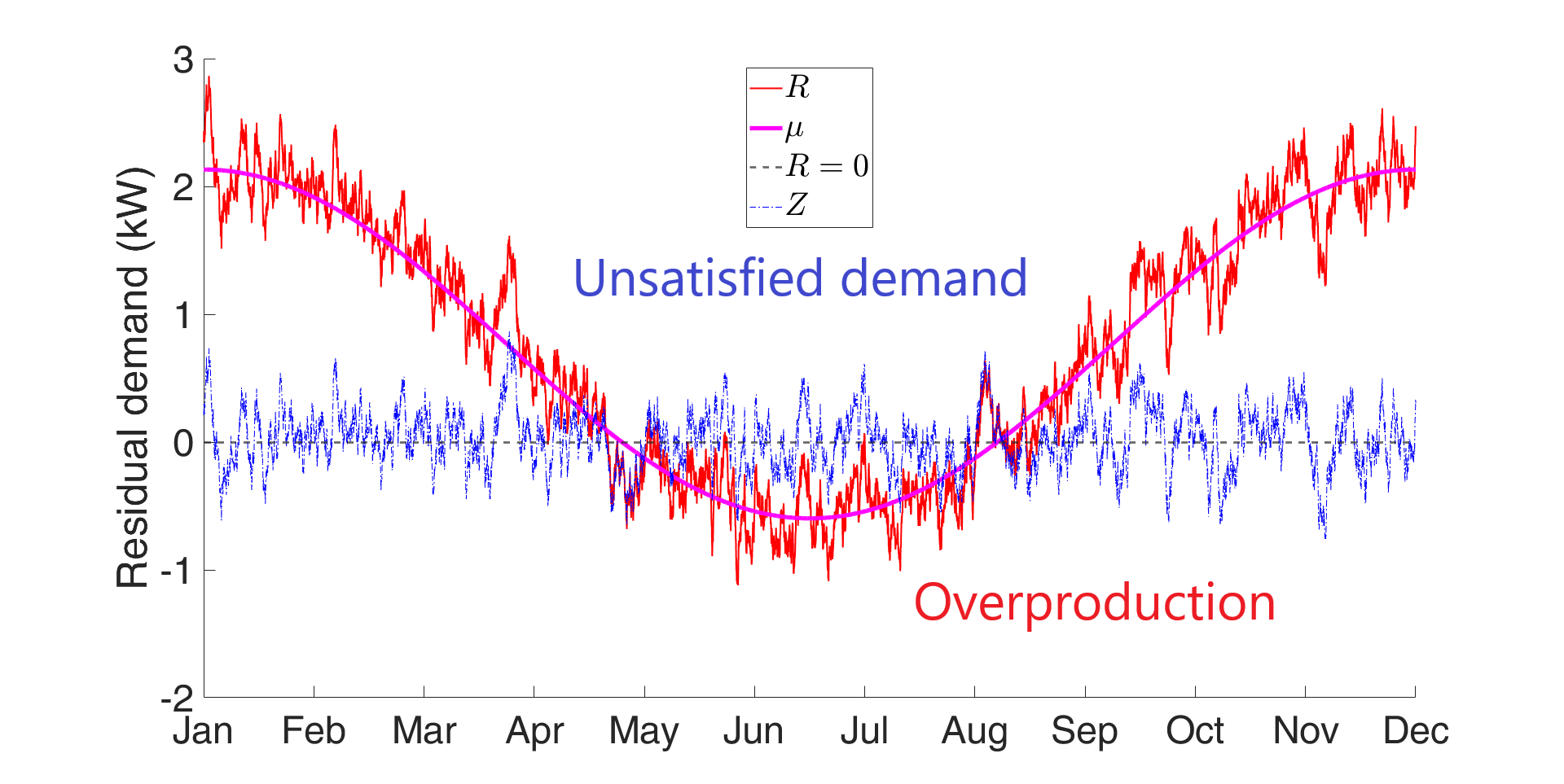}
	\caption[Residual demand.]{Residual demand.}
	\label{residfig}
\end{figure}

\begin{remark} \quad \\[-4ex]
	\begin{enumerate}
		\item From Equation \eqref{residdemand}, we deduce that  $\mathbb{E}[R(t)] \rightarrow \muR(t)$ as $t\rightarrow \infty$, implying that $\muR(t)$ is the long term mean of the residual demand.
		\item From Equation \eqref{residdemand}, we also deduce that $ \ZR(t) = R(t) - \muR(t)$, further implying that $\ZR$ is the deviation of the residual demand $R$ from the long term mean $\muR$. Hence the justification of the terminology deseasonalized residual demand.
		\item In the following, the formulation in Equation \eqref{residdemand} allows to work with general seasonality functions $\muR$.
	\end{enumerate}
\end{remark}

\subsection{External storage}
In the following, we consider as storage a water tank where the water mass and specific heat capacity are given by $m_Q [kg]$ and $c_{P} [KWh/kg~K]$, respectively. The surface area of the tank is denoted as $A [m^2]$. \newline 
Given the temperature $Q$ in the external storage, the amount $E$ of thermal energy stored in the external storage is given by
\begin{equation}
	\label{thermalenergy1}
	E(t) = m_Qc_PQ(t).
\end{equation} 
From the heat fluxes and the Newton cooling law with heat transfer coefficient  $\gamma$ [$kW/m^2 K$] and ambient temperature around the ES $Q_{amb}$ [$^{\circ}$C], the change in the thermal energy of the water inside the ES is given by
\begin{equation}
	\label{thermalenergy2}
	\mathrm{d}E(t) = -[(1-\alpha(t))R(t) + A\gamma(Q(t)-Q_{amb}(t))]\mathrm{d}t, \quad E(0) = E_0.
\end{equation}
The first term $(1-\alpha(t))R(t)$ represents the ES charging or discharging rate. Indeed for an unsatisfied demand ($R(t)>0$) and for a non-empty ES ($Q(t)>q_{min}$), the prosumer discharges the storage at a rate $(1-\alpha(t))R(t)>0$. For an overproduction ($R(t)<0$) and for a non-full ES ($Q(t)<q_{max}$), the prosumer can charge the storage at the rate $(1-\alpha(t))R(t)<0$. The second term $A\gamma (Q(t)-Q_{amb}(t))$ is the rate of heat transfer with the environment. \newline
From Equations \eqref{thermalenergy1} and \eqref{thermalenergy2}, the dynamics of the temperature in the ES is given by
\begin{equation}
	\label{es}
	\mathrm{d}Q(t) = \frac{-1}{m_Qc_{P}}[(1-\alpha(t))R(t)+A\gamma (Q(t)-Q_{amb}(t))]\mathrm{d}t, \quad Q(0)=Q_{0}\in [q_{min},q_{max}].
\end{equation}
\begin{remark}\quad \\[-4ex]
	\begin{itemize}
		\item 	 In Equation \eqref{es}, the ambient temperature $Q_{amb}(t)$ is modelled as time dependent to capture the changes of this predictable and slowly varying ambient quantity. A large $\gamma$ implies large losses per unit of time prohibiting long term storage. Similarly, a small $\gamma$ implies small losses per unit of time allowing long term storage. The calibration of $\gamma$ is provided in Appendix \ref{calibration_gamma}.
		\item In the case $\alpha = 0$, the prosumer does not satisfy its residual demand from the CHS. In the event of an unsatisfied demand ($R(t)>0$), $\alpha = 1$ implies that the prosumer purchases all the residual demand from CHS. However, for an overproduction ($R(t)<0$), $\alpha = 1$ implies that the prosumer sells all the residual demand to CHS.
	\end{itemize}
\end{remark}

\begin{figure}[ht!]			
	\centering
	\includegraphics[width=1.0\textwidth]{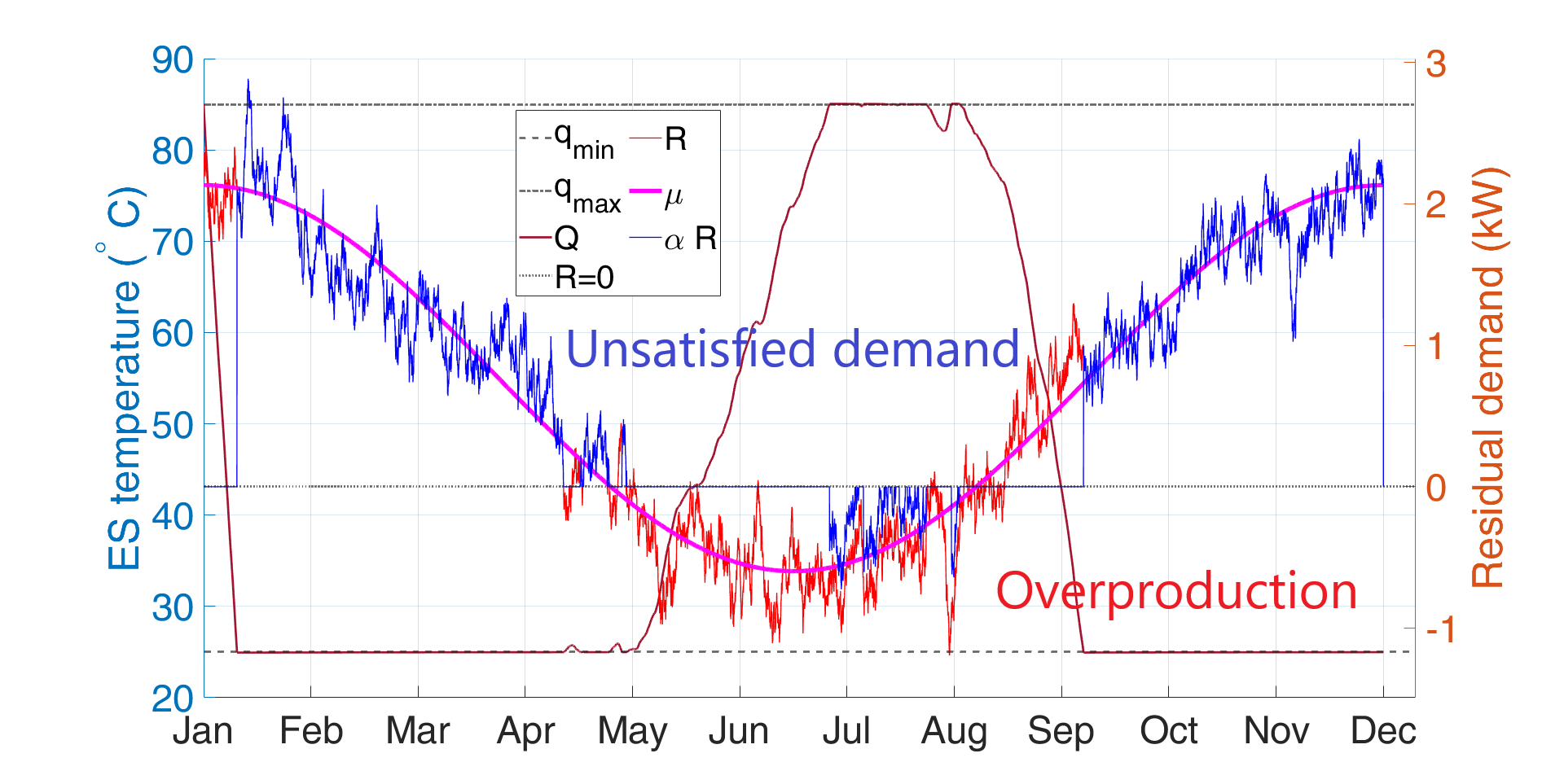}
	\caption[Temperature in external storage.]{Temperature in external storage with residual demand and a given control .}
	\label{EStemp}
\end{figure}

In Figure \eqref{EStemp}, we consider the control use the ES as much as possible. In this formulation, for an unsatisfied demand, we discharge ES as long as it is non-empty (i.e. $Q(t)>q_{min}$). For an overproduction we charge ES and compensate for the loss to the environment when $Q(t)=q_{max}$, i.e. the storage is full.

\subsection{Internal storage}
In our modelling perspectives, the internal storage is a cylinder allowing minimum and maximum temperatures $\pi_{min}$ and $\pi_{max}$, respectively and serves the purpose of a short-term storage of thermal energy. It is filled with water which is later heated from the solar collector to produce thermal energy in order to satisfy the prosumer's heat and hot water demand. 
In the following, we do not model the dynamics of the internal storage, rather we treat it as a ``black box".

\begin{remark}
	Possible models for the internal storage include an ODE when considering a non-stratified storage and a PDE for the case of a stratified storage.
\end{remark}

\subsection{Heat price}
We distinguish time-depdendent buying and selling prices denoted respectively by $S_{buy}(t)$ and $S_{sell}(t)$ for all $t\in [0,T]$. We assume an underlying contract between CHS and prosumer binding the former to always satisfy both the unsatisfied demand or overproduction of the latter. As an incentive, CHS has the right to set $S_{buy}(t)$ and $S_{sell}(t)$.\newline 
In this project, we model a bid-ask spread for the heat price as
\begin{equation}
	\label{heatprice}
	S_{buy}(t) = L^S_{0}+ \sum^{k}_{i=1} L^S_i\cos\Big(\frac{2\pi}{\rho^S_i}(t-t^S_i)\Big), \quad \text{and} \quad S_{sell}(t) = S_{buy}(t) - \xi,
\end{equation}
where $\xi>0$ denotes the spread, $L^S_{0}$ $[\EUR/kWh]$ the long term mean heat buying price, $L^S_i$ $[\EUR/kWh]$, $\rho^S_i$ $[h]$ and $t^S_i$ $[h]$ represent respectively, the amplitude, length and a reference time of the seasonality component $i$. $k$ represents the number of seasonality components. In the above formulation, we assume that heat price can be influenced by yearly, hourly and daily components for example. 
\begin{remark} In the following we discuss the above heat price formulation.
	\begin{enumerate}
		\item If CHS charges too high a buying price, the prosumer can escape CHS' market power by using ES ``as much as possible" and (or) diversify its thermal energy production sources as well increasing the production capacity of the solar collector. However charging too low a buying price to avoid this situation could hinder its operational abilities.
		\item The above price formulation avoids arbitrage opportunities, since if $S_{sell}(t) > S_{buy}(t) + \xi$ for all $t\in [0,T]$, a prosumer could purchase heat at $S_{buy}(t)$ and immediately sell it for $S_{sell}(t)$. In this scenario, we can interpret $\xi$ as transaction cost (cost of electricity consumption) during buying and selling.
		\item Though we considered $S_{buy}(t)>0$, one can also allow a negative heat buying price $S_{buy}(t)$. This scenario occurs in the event the CHS' storage is full yet most (probably all) prosumers do not purchase heat from CHS due to their individual overproduction. In this case, since the CHS sets the heat buying, it has the freedom substantially decrease $S_{buy}(t)$ to incentivise prosumers to purchase its heat.\newline
		A similar situation is observed for electricity prices in power markets.
	\end{enumerate}
\end{remark}
The state of the control system at time $t\in [0,T]$ is given by
\begin{itemize}
	\item $\ZR(t)$ denoting the deseasonalized residual demand [kW] and
	\item $Q(t)$ representing the temperature level in ES [$^{\circ}$C].
\end{itemize}
\subsection{Constraints}
In the following we formulate and discuss both state and control constraints. The first are imposed, while the latter emanate physically from our modelling perspective.

\subsubsection{State constraints} We impose physical constraints to the states of our system. The ES  is assumed to have a maximum and minimum storage capacity respectively denoted by $q_{max}$ and $q_{min}$. Thus we derive the condition $Q(t)\in [q_{min}, q_{max}]$. A fully charged (or full) ES is represented by $Q(t) = q_{max}$ and an empty storage by $Q(t) = q_{min}$. 

\subsubsection{Control constraints} From the system's requirements, we formulate control constraints which are derived using the different scenarios $Q(t) = q_{min}$ and $Q(t) = q_{max}$ coupled with the sign of the residual demand, $R(t)>0$ (unsatisfied demand) and $R(t)<0$ (overproduction). The operational constraints the prosumer faces, lead to state-dependent control constraints.
\begin{enumerate}
	\item For an empty ES (i.e. $Q(t) = q_{min}$), no discharging is allowed, only charging. Thus if:
	\begin{itemize}
		\item $R(t) > 0$, then $(1-\alpha(t))R(t) = 0$ implying $\alpha(t)=1$,
		\item $R(t) < 0$, no further restriction than $\alpha(t)\in [0,1]$.
	\end{itemize}
	\item For a full ES (i.e. $Q(t)=q_{max}$), we consider the following scenarios.
	\begin{itemize}
		\item $R(t) > 0$, there is no further restriction than $\alpha(t)\in [0,1]$,
		\item $R(t) < 0$ and Equation \eqref{es} imply that $(\alpha(t)-1)R(t)-A\gamma (q_{max}-Q_{amb}(t))\le 0$. \newline
		Thus $\alpha(t) \ge 1 + \frac{A\gamma (q_{max}-Q_{amb}(t)}{R(t)} =: \widetilde{c}(t,R(t))$ This implies that $1-\widetilde{c}(t,R(t)) \ge 0$ is the maximal proportion of relative demand $R(t)$ which can still be stored in the ES to compensate for the loss to the environment. If $1-\widetilde{c}(t,R(t)) \ge 1$, the complete overproduction $R(t)$ can be stored. In the following, we define $c^*(t,R(t)) = \widetilde{c}(t,R(t))\vee 0$ and choose $\alpha(t) = c^*(t,R(t))$. 
	\end{itemize} 
\end{enumerate}
In light of the above, we obtain the following			
\[\alpha(t) \in \begin{cases}
	\begin{array}{ll}
		[0,1], & \text{for~} Q(t)>q_{min} ~\quad \ZR(t) \ge - \muR(t), \\
		\{1\}, & \text{for~} Q(t)=q_{min}, \quad \ZR(t) \ge - \muR(t),
	\end{array}  
\end{cases}
\]

\[\quad ~~\alpha(t) \in \begin{cases}
	\begin{array}{ll}
		[0,1], & \text{for~} Q(t)<q_{max} \quad \ZR(t) < - \muR(t), \\
		[c^*(t),1], & \text{for~} Q(t)=q_{max},\quad \ZR(t) < - \muR(t). 
	\end{array}
\end{cases}
\]
We denote by $\mathcal{K}_a$ the state dependent set of feasible controls defined by

\[ \mathcal{K}_a(t,\zR,q) =  \begin{cases}
	\begin{array}{ll}
		[0,1], & q>q_{min}, \quad \zR \ge - \muR(t), \\
		\{1\}, & q=q_{min}, \quad \zR \ge - \muR(t), \\
		[0,1], & q<q_{max},\quad \zR < - \muR(t), \\
		[c^*(t),1], & q=q_{max},\quad \zR < - \muR(t). 
	\end{array} 
\end{cases}\]
\begin{figure}[ht]			
	\centering
	\includegraphics[width=1\textwidth]{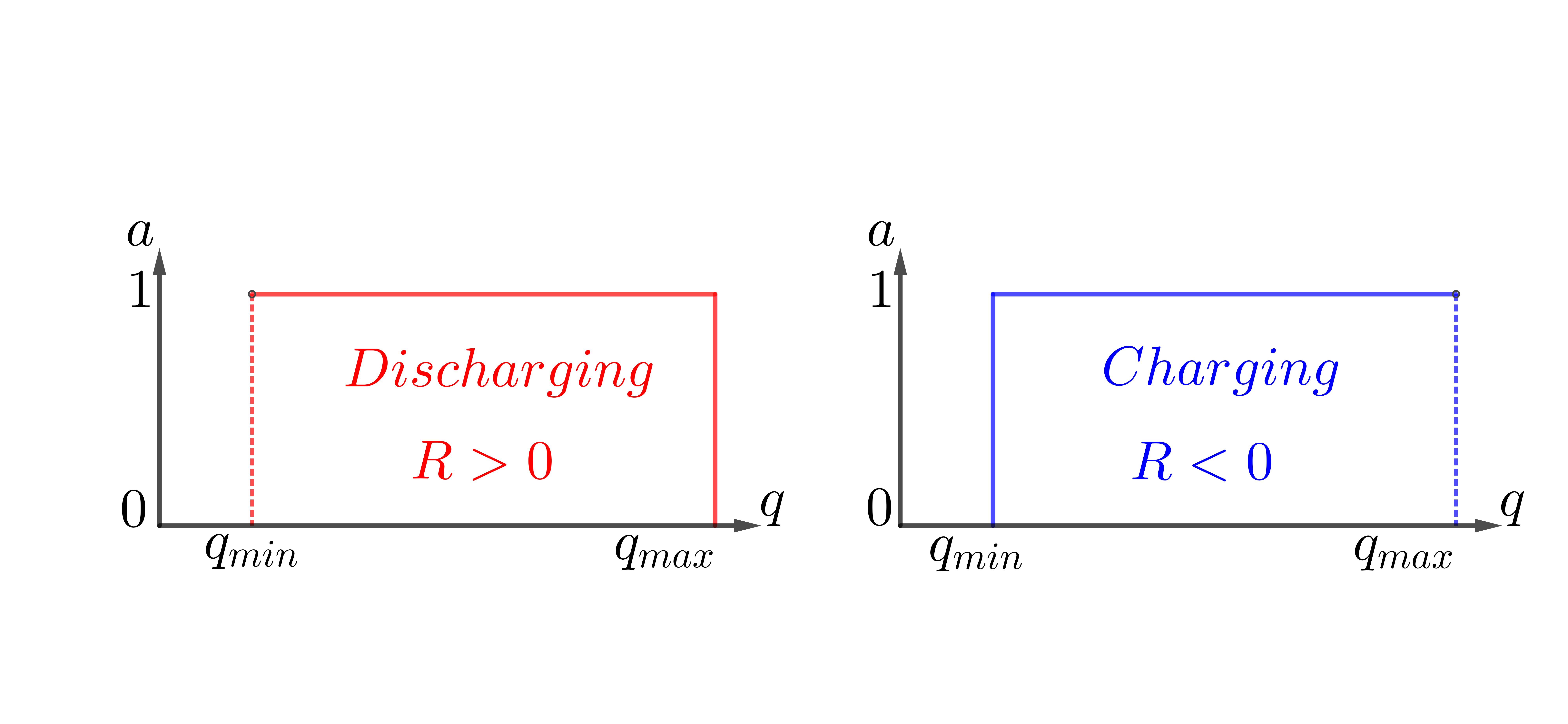}
	\vspace{-1.5cm}
	\caption{Set of feasible controls $\mathcal{K}_a$ with unsatisfied demand (Left) and overproduction (Right) for an imperfect insulation and non-constant seasonality.}
	\label{Nctsfeassets}
\end{figure}

\paragraph{Interpretation}
\begin{enumerate}
	\item In the case of an unsatisfied demand, the prosumer can either purchase thermal energy from the CHS or discharge ES.
	\begin{itemize}
		\item[-] If ES is empty (i.e. $Q(t)=q_{min}$), no discharging is allowed. Then the prosumer satisfies all residual demand from the CHS.
		\item[-] If $Q(t)>q_{min}$, the prosumer has the option to either discharge ES or purchase thermal energy from CHS. This choice depends on the heat buying price and the cost of electricity to run the pumps (heat and ordinary).
	\end{itemize}
	\item In the event of an overproduction, the prosumer either stores excess energy in the ES or sells overproduction to the CHS for a revenue.
	\begin{itemize}
		\item[-] If ES is not properly filled, i.e. $Q(t)<q_{max}$, the prosumer either stores the excess in ES or sells to CHS depending on the heat selling price.
		\item[-] If ES is full, thus the prosumer can charge but only for compensating the loss due to the heat transfer to the environment, or sell excess to CHS. 
	\end{itemize}
\end{enumerate}
%
	
	\section{Stochastic optimal control}
	\label{soc}
	In this section we formulate a stochastic optimal control problem for a representative prosumer. The objective of this prosumer is to satisfy its demand for heating and hot water while minimizing payments to CHS and the cost of electricity consumption to operate the pumps.\newline
	We formulate a mathematical model to find the value and the optimal strategy for a residential heating system and consider an associated stochastic optimal control problem with state process $X=(\ZR,Q)$. From Equations \eqref{residualdemand} and \eqref{es}, the state process follows the dynamics
	\begin{align}
		\label{states}
		\mathrm{d}\ZR(t) &= f(\ZR(t))\mathrm{d}t + \sigmaR\mathrm{d}\WR(t), \qquad \quad \ZR(0)=\ZR_0\in \mathbb{R}, \nonumber\\
		\mathrm{d}Q(t) &= g(\ZR(t),Q(t),\alpha(t))\mathrm{d}t, \qquad \quad Q(0) = Q_{0}\in [q_{min},q_{max}],
	\end{align}
	where $f(\zR) = -\kappaR \zR$ and $g(\zR,q,a)=\frac{-1}{m_Qc_{P}}((1-a)(\muR(t)+\zR)+A\gamma (q-Q_{amb}))$.\newline
	Note that only the storage level $Q$ is controlled, but not the residual demand $\ZR$ and we shall write $Q^\alpha$ instead of $Q$ to indicate this dependence.
	\subsection{Running costs}
	The running costs formulation comprises several components made of the cost (revenue) of buying (selling) heat from (to) CHS, and the cost of running the heat and ordinary pumps.\newline
	To satisfy its residual demand, the prosumer can purchase heat from CHS for a price $S_{buy}(t)$ or sell heat to CHS for a price $S_{sell}(t)$ for $t\in[0,T]$.\newline
	The cost (revenue) per unit of time of buying (selling) heat from (to) CHS denoted $C_0$ is defined by
	\[ C_0(t,\zR,a) = \left\{
	\begin{array}{ll}
		a(\muR(t)+\zR)S_{buy}(t), \qquad & \zR \ge - \muR(t), \\
		a(\muR(t)+\zR)S_{sell}(t), \qquad & \zR < - \muR(t).
	\end{array} 
	\right. \]
	In order to satisfy its demand from CHS for $R(t)\ge 0$, the prosumer uses a heat pump to increase the temperature of water contained in the connecting pipe, hence generating a cost due to electricity consumption.\newline
	Let $S_{el}$ denote the elctricity price, $\pi_{min}$ the minimum temperature in IS and $\pi_d$ the outlet temperature of the heat pump. \newline 
	\begin{assumption}
		\label{C1ineq}
		We assume that the condition $q_{max} > \pi_d > \pi_{min} > P_c$ holds.
	\end{assumption}
	Under Assumption \ref{C1ineq}, the cost $C_1$ of using the heat pump ,  can be defined by
	\[ C_1(t,\zR,a) = \left\{
	\begin{array}{ll}
		a(\muR(t)+\zR)(d_1+d_2(\pi_d-P_c))S_{el}, & \zR \ge - \muR(t), \\
		-a(\muR(t)+\zR)d_1S_{el}, & \zR < - \muR(t),
	\end{array} 
	\right. \]
	where $d_1$ and $d_2~[K^{-1}]$ are positive constants with their calibration provided in Appendix \ref{calibrationd1_d2}. $d_1$ penalizes a high water flow rate in the pipe, meaning that a high $d_1$ implies a high flow rate of water. $d_2$ penalizes a high temperature difference between $P_c$ and $\pi_d$. A high $d_2$ means that more electricity is consummed to raise the temperature from $P_c$ to $\pi_d$. For $R(t) \ge 0$, the formulation of $C_1$ takes into consideration both the consumption of electricity to transport the heat and increase the temperature at the desired level $\pi_d$. In the event $R(t) < 0$, $C_1$ takes only into consideration the cost of electricity consumed to transport the heat to the CHS' storage since no increase of temperature is required.\newline
	In satisfying its demand from ES, the prosumer uses an ordinary pump for the heat transfer, hence generating a cost $C_2$ also due to its electricity consumption. This cost of using an ordinary pump during heat transfer between IS and ES is defined by
	\begin{equation*}
		C_2(t,\zR,a) = (1-a)|\muR(t)+\zR|d_1S_{el}.
	\end{equation*}
	In the above formulation, the functions $C_0,~C_1$ and $C_2$ depend on time $t$ through the seasonality function.\newline
	Hence the running cost defined by $\Psi(t,\zR,a) = C_0(t,\zR,a) + C_1(t,\zR,a) + C_2(t,\zR,a)$ can be represented as
	\begin{equation}
		\label{psi_P}
		\Psi(t,\zR,a) =\left\{
		\begin{array}{ll}
			(\muR(t)+\zR)(a[S_{buy}(t) + d_2(\pi_d-P_c)S_{el}] + d_1S_{el}), \quad & \zR \ge - \muR(t) \\ \\
			(\muR(t)+\zR)(aS_{sell}(t)-d_1S_{el}), \quad & \zR < - \muR(t).
		\end{array} 
		\right.
	\end{equation}
	\begin{remark}
		We assume a constant electricity price since our prosumer is a small consumers of electricity and as such do not pay the randomly fluctuating market price, but some fixed tariff.
	\end{remark}
	\subsection{Terminal cost}
	\label{Terminalcost}
	In our model with finite time horizon $T$, we also consider a final cash flow depending on ES storage level $Q^\alpha(T)$ given by the function $\Phi(Q^\alpha(T))$. In our case, this function includes a storage contract with penalties for failing to leave ES at a prescribed level. Let $q_{crit}$ denote a critical temperature level. Further, let the liquidation price $S_{liq}$ satisfy $S_{liq} < S_{sell}(T)$, and the penalty price $S_{pen}$ also satisfy $S_{pen} > S_{buy}(T)$. Then, we formulate the terminal cost as 
	\begin{equation}
		\label{phi_P}
		\Phi(q) =
		\left\{ \begin{array}{ll}
			S_{pen}m_Qc_P(q_{crit}-q), & q<q_{crit},\\
			-S_{liq}m_Qc_P(q-q_{crit}), & q\geq q_{crit}. 
		\end{array} \right.
	\end{equation}
	\begin{remark}
		From Equation \eqref{phi_P}, one can easily recover similar cases of \textit{ storage expires worthless}, \textit{liquidation of the storage} and \textit{penalty payment} as in \cite{shardin2017partially}. 
	\end{remark}
		Given a control process $\alpha = (\alpha_t)_{t\in[0,T]}$, the performance criterion $J(t,x;\alpha)$ is the expected aggregated total discounted cash flow over the time interval $[0,T]$ defined by
		\begin{equation}
			\label{perfunct}
			J(t,x,\alpha) = \mathbb{E}_{t,x}\Big[\int_t^T e^{-\delta s}\Psi(s,\ZR(s),\alpha(s))
			\mathrm{d}s + e^{-\delta T}\Phi(Q^\alpha(T))\Big],
		\end{equation}
		for $x = (\zR,q)\in \mathbb{R}\times [q_{min},q_{max}]$. 
		$\delta\ge 0$ denotes the discount rate and $\mathbb{E}_{t,x}[\cdot]$ is the conditional expectation given that at initial time $t$, $X(t)=x$.\newline
		A progressively measurable control process $\alpha$ w.r.t the filtration $\mathbb{F}$, of Markov type, i.e.,
		$\alpha(t) = \tilde{a}(t,X(t))$ for all $t\in [0,T]$ with measurable function $\tilde{a}$, and satisfying the control constraints $\alpha(t)\in \mathcal{K}_a(t,X(t))$ is called admissible. The class of admissible controls is denoted by $\mathcal{A}$ and defined as
		\begin{align}
			\mathcal{A} &= \{(\alpha(t))_{t\in [0,T]} ~| \quad \alpha ~\text{is}~ \mathbb{F}-\text{progressively measurable}, \alpha(t) = \tilde{a}(t,X(t)),\nonumber\\ &\qquad \forall t\in [0,T], \tilde{a}(t,x)\in \mathcal{K}_a(t,x) ~\text{and ~Equations~} \eqref{states} ~\text{and~} \eqref{perfunct} ~\text{are~well~defined}\}.
		\end{align}
		The prosumer's objective is to minimize the performance function over all admissible controls in $\mathcal{A}$. We define the value function for all $(t,x)\in [0,T]\times \mathbb{R}\times [q_{min},q_{max}]$ by
		\begin{equation}
			V(t,x) = \inf_{\alpha\in \mathcal{A}}J(t,x,\alpha).
		\end{equation}
		A control $\alpha^*\in \mathcal{A}$ satisfying $V(t,x) = J(t,x,\alpha^*)$ is called optimal control.\newline
		Let $h>0$ be such that $t+h<T$. As in \cite{pham2009continuous}, we can derive the following dynamic programming principle (DPP)
		\begin{align}
			\label{dpp}
			V(t,x) &= \inf_{\alpha\in \mathcal{A}}\mathbb{E}_{t,x}\Big[\int_{t}^{t+h} e^{-\delta (s-t)}\Psi(t,\ZR(s),\alpha(s))
			\mathrm{d}s \nonumber\\
			&\hspace{2.5cm} + e^{-\delta h}V(t+h,R(t+h),Q^\alpha(t+h))\Big].
		\end{align}
		\section{Discrete-time numerical scheme}
		We motivate a finite difference scheme for the solution of the optimal control problem based on the theory developed in \cite{duffy2013finite}. For computational purposes we truncate the domain $\mathcal{X} = \mathbb{R}\times [q_{min},q_{max}]$ to $\overline{\mathcal{X}} = \mathcal{S}\times [q_{min},q_{max}]$, where $\mathbb{R}$ is replaced by a closed interval $\mathcal{S} = [\zR_{min},\zR_{max}]$ such that for a given tolerance $\epsilon \ll 1$, $\mathbb{P}(\ZR(t)\in \mathcal{S})\ge 1-\epsilon$ for all $t\in [0,T]$, where $\zR_{min}$ and $\zR_{max}$ are respectively, the minimum and maximum residual demand.\newline
		With $\frac{\sigmaR}{\sqrt{2\kappaR}}$, the asymptotic standard deviation of the deseasonalized OU-process $\ZR$ and using the $c_\epsilon\sigma$-rule, we have $\zR_{min} = -c_\epsilon \frac{\sigmaR}{\sqrt{2\kappaR}} ~\text{and~} \zR_{max} = c_\epsilon \frac{\sigmaR}{\sqrt{2\kappaR}}$, where $c_\epsilon = 3,4,\text{~or~}5$.\newline
		Let $N_t, N_{\zR}$ and $N_q$ denote the mesh sizes in the $t,\zR,q$-directions respectively. Define
		\begin{align*}
			\mathcal{G} &= \mathcal{G}_t\times \mathcal{G}_{\zR} \times \mathcal{G}_q = [t_0, \ldots,t_{N_t}]\times [\zR_0,\ldots,\zR_{N_{\zR}}]\times [q_0,\ldots,q_{N_q}],
		\end{align*}
		where $	t_n = t_0+n\Delta t$, $z_{\ell} = \zR_{min}+\ell\Delta \zR$, $q_j = q_{min}+j\Delta q$, for $n\in \mathcal{N}=\{0,1,\ldots,N_t\}$, $\ell=0,\ldots,N_{\zR}$ and $j=0,\ldots,N_q$, as a 3-dimensional equidistant grid on $\overline{\mathcal{X}}$ with the temporal and spatial step sizes 
		\begin{equation*}
			t_{n+1}-t_n=:\Delta t = \frac{T}{N_t},\quad \Delta q = \frac{q_{max}-q_{min}}{N_q}, \quad \Delta \zR = \frac{\zR_{max}-\zR_{min}}{N_{\zR}}.
		\end{equation*}
		In the sequel, we formulate a discrete-time scheme starting from DPP. It is an alternative approach to the semi-Lagrangian scheme in \cite{chen2008semi,ware2013accurate} which allows us to avoid using the time reversal technique. The idea for the derivation of the scheme consists of 6 parts. We start by fixing the triple $(t_n,\zR_\ell,q_j)$ and assume that for $t\in[t_n,t_{n+1})$, the stochastic state is fixed and known, i.e. $\ZR(t) = \zR_{\ell}$ and further assume that the optimal control is also fixed and unknown, i.e. $\alpha(t) = \alpha^n_{\ell,j}$. Next, for the fixed and known $\zR_{\ell}$ and letting $a=\alpha^n_{\ell,j}$ we define the path of the temperature $\Q^{a}(t)$, for $t\in [t_n,t_{n+1})$. Using Euler's discretization, we obtain the approximation $\Q^{a,n+1}_{j(\ell,n)}$. This is the arrival point of the ES temperature level at time $t_{n+1}$ knowing that at time $t_n$, the ES temperature level was at $q_j$ for a fixed deseasonalized residual demand $\zR_{\ell}$. Now, from the DPP and under a given assumption, we derive the approximate optimal control $\alpha^{*n}_{\ell,j}$. Afterwards, we "release" the stochastic state, i.e. $\ZR$ is now stochastic and substitute $\alpha^{*n}_{\ell,j}$ in DPP to remove the infimum term. Applying the Feynman-Kac's formula, we obtain a terminal condition problem in terms of the value function which we discretize and form a linear system of equations in order to get the value functions. 
		Assumption \ref{AsumptrProsNoCN} below is useful in deriving the approximation $\Q^{a,n+1}_{j(\ell,n)}$ which is important in formulating the discrete scheme starting from the DPP. 
		\begin{assumption}
			\label{AsumptrProsNoCN}
			We assume that for $t\in [t_n,t_{n+1})$, $a = \alpha^n_{\ell, j}$ is fixed but unknown, and further assume that for $t\in [t_n,t_{n+1})$, $Z(t) = \zR_\ell$. In addition, for $t\in [t_n,t_{n+1})$, we assume that $\muR(t)$ and $Q_{amb}(t)$ are piecewise constant, i.e.
			$\muR(t) = \displaystyle \sum_{n=0}^{N_t-1}\muR^n \mathbbm{1}_{[t_n,t_{n+1})}(t), \quad \muR(T) = \muR^N$ and $Q_{amb}(t) = \displaystyle \sum_{n=0}^{N_t-1}Q_{amb}^n \mathbbm{1}_{[t_n,t_{n+1})}(t), \quad Q_{amb}(T) = Q_{amb}^N$.
		\end{assumption}
		Under Assumption \ref{AsumptrProsNoCN}, the path of the temperature of the external storage, $\Q^{a}$ satisfies
		\begin{align}
			\label{pathES}
			\mathrm{d}\Q^{a}(t) &= g(z_{\ell},\Q^{a}(t),a)\mathrm{d}t, \quad \Q^a(t_n) = q_j,~ Q_{amb}(t_n) = Q_{amb}^n.
		\end{align}
		Let $\Q^{a,n+1}_{j(\ell,n)}$ denote an approximation of $\Q^{a}(t_{n+1})$. Applying Euler's approximation to Equation \eqref{pathES}, we obtain
		\begin{equation}
			\label{dSLT}
			\Q^{a,n+1}_{j(\ell,n)}:= \Big(1-\frac{A\gamma \Delta t}{m_Qc_P}\Big)q_j + \frac{\Delta t}{m_Qc_P}[(a-1)(\muR^n+\zR_\ell) + A\gamma Q_{amb}^n].
		\end{equation}
		\subsection{State-dependent control contraints}
		Recall that $a$ must satisfy the constraint $a \in \mathcal{K}_a(t,\zR_{\ell},q_j)$. Hence to prevent values of $\Q^{a,n+1}_{j(\ell,n)}$ from lying outside the domain $[q_{min},q_{max}]$, we reformulate $\mathcal{K}_a(t,\zR,q)$ to adapt to the discrete-time setting
		\begin{equation}
			\label{feassetapprox}
			\mathcal{K}^n_d(\zR_{\ell},q_j) = \{a\in \mathcal{K}_a(t,\zR_{\ell},q_j) \quad | \quad \Q^{a,n+1}_{j(\ell,n)}\in [q_{min},q_{max}]\}.
		\end{equation}
		In the following, we seek to fully describe the set $\mathcal{K}^n_d(\zR_{\ell},q_j)$. First, from the requirement $\Q^{a,n+1}_{j(\ell,n)}\in [q_{min},q_{max}]$, we derive the lower and upper bounds for the control $a$ in the following lemma.
		\begin{lemma}
			\label{controlineq}
			Let $Z(t_n)=\zR=\zR_\ell$ and $\Q^{a}(t_n)=q=q_j$. Then the condition $\Q^{a,n+1}_{j(\ell,n)}\in [q_{min},q_{max}]$ is fulfilled if the control $a$ satisfies the relation
			\begin{equation}
				\underline{\chi}(n,\zR,q)\le a \le \overline{\chi}(n,\zR,q),
			\end{equation}
			\begin{align*}
				\underline{\chi}(n,\zR,q) &= 1 - \frac{m_Qc_P(q-q_{min}) - A\gamma\Delta t(q-Q_{amb}^n)}{|\muR^n+\zR| \Delta t}, \\
				\overline{\chi}(n,\zR,q) &= 1 + \frac{m_Qc_P(q_{max}-q) + A\gamma\Delta t(q-Q_{amb}^n)}{|\muR^n+\zR| \Delta t}.
			\end{align*}
			\end{lemma}
			\begin{proof}
				Suppose $\Q^{a,n+1}_{j(\ell,n)} \ge q_{min}$. Then from Equation \eqref{dSLT}, solving for $a$ yields $a\ge 1 -\frac{m_Qc_P}{|\muR^n+\zR_\ell|\Delta t}(q_j-q_{min}) + \frac{A\gamma}{|\muR^n+\zR_\ell|}(q_j-Q^n_{amb})=: \underline{\chi}(\muR^n,\zR,q)$. Similarly $\Q^{a,n+1}_{j(\ell,n)} \le q_{max}$ implies $a \le 1 + \frac{m_Qc_P}{|\muR^n+\zR_\ell|\Delta t}(q_{max}-q_j) + \frac{A\gamma}{|\muR^n+\zR_\ell|}(q_j-Q^n_{amb})=: \overline{\chi}(\muR^n,\zR,q)$.
			\end{proof}
			Let $q^n_d(\zR)$ be the temperature level such that in the event of an unsatisfied demand, a subsequent discharge of ES with maximum rate i.e., $a=0$, will drive it empty in the next time step. Similarly, let $q^n_u(\zR)$ be the temperature level such that for an overproduction, the next charge of ES with maximum rate, yields a full storage at the next time step.\newline
			Below, we derive $q^n_d(z)$ and $q^n_u(z)$ which is useful to adequately describe $\mathcal{K}^n_d$. \newline 
			Now, setting $a=0$ in Equation \eqref{pathES} together with assuming $\Q^a(t_n)=q^n_d(\zR)$ and $\Q^a(t)=q_{min}$, we obtain
			\begin{equation*}
				q^n_d(\zR) = \frac{m_Qc_Pq_{min} - \Delta t(A\gamma Q_{amb}^n-\muR^n-\zR)}{m_Qc_P-A\gamma\Delta t}.
			\end{equation*}
			In the following, we discuss the formulation of $q^n_u(\zR)$. We start by introducing the ``compensation charging rate", $r^*(q)$ which is the residual demand such that when injected in ES, its temperature remains unchanged, i.e. $\mathrm{d}\Q^a(t)=0$. In order to derive $r^*(q)$, we substitute $a=0$ in Equation \eqref{pathES} to obtain $r^*(q) = -A\gamma (q-Q_{amb}^n)$. Indeed, considering $r^*(q)$ allows us to distinguish between ``strong" and ``small" overproduction.\newline 
			We assume that for a ``strong" overproduction, i.e., $\muR^n+\zR\le r^*(q)$, no more heat can be charged in ES and thus set $q^n_u = q_{max}$. However, for a ``small" overproduction, i.e., $\muR^n+\zR > r^*(q)$, the residual demand $\muR^n+\zR$ can be charged in ES without increasing the temperature level in ES. In this case, we set $a=0$ in Equation \eqref{pathES} together with assuming $\Q^a(t_n)=q^n_d$ and $\Q^a(t)=q_{max}$, to obtain
			\begin{equation*}
				\hat{q}^n(\zR) = \frac{m_Qc_Pq_{max}-\Delta t(A\gamma Q_{amb}^n-\muR^n-\zR)}{m_Qc_P-A\gamma\Delta t}.
			\end{equation*}
			Therefore the temperature $q^n_u$ is given by
			\[q^n_u(\zR) =  \begin{cases}
				\begin{array}{ll}
					q_{max}, & \muR^n+\zR\le r^*(q), \\
					\hat{q}^n(\zR), & \muR^n+\zR > r^*(q).
				\end{array} 
			\end{cases}\]
	Hence, it follows that the state-dependent set of feasible controls $\mathcal{K}_d(\zR,q)$ is given by
			\[ \mathcal{K}^n_d(\zR,q) =  \begin{cases}
				\begin{array}{ll}
					[0,1], & q\ge q^n_d(\zR), \quad \zR \ge - \muR^n \\
					[x^n_1, 1], & q<q^n_d(\zR), \quad \zR \ge - \muR^n \\
					[0,1], & q < q^n_u(\zR),\quad \zR < - \muR^n \\
					[x^n_2,1], & q \ge q^n_u(\zR),\quad \zR < - \muR^n, 
				\end{array} 
			\end{cases}\] 
			where $x^n_1(\zR,q) = \max \Big(0,1-\frac{m_Qc_P}{(\muR^n+\zR)\Delta t}(q-q_{min})+\frac{A\gamma}{(\muR^n+\zR)}(q-Q_{amb}^n)\Big)$\newline
			and $\quad x^n_2(\zR,q) = \max \Big(1+\frac{m_Qc_P}{(\muR^n+\zR)\Delta t}(q_{max}-q)+\frac{A\gamma}{(\muR^n+\zR)}(q-Q_{amb}^n),0\Big)$.
			
			\begin{figure}[ht]			
				\centering
				\vspace{-0.5cm}
				\includegraphics[width=1\textwidth]{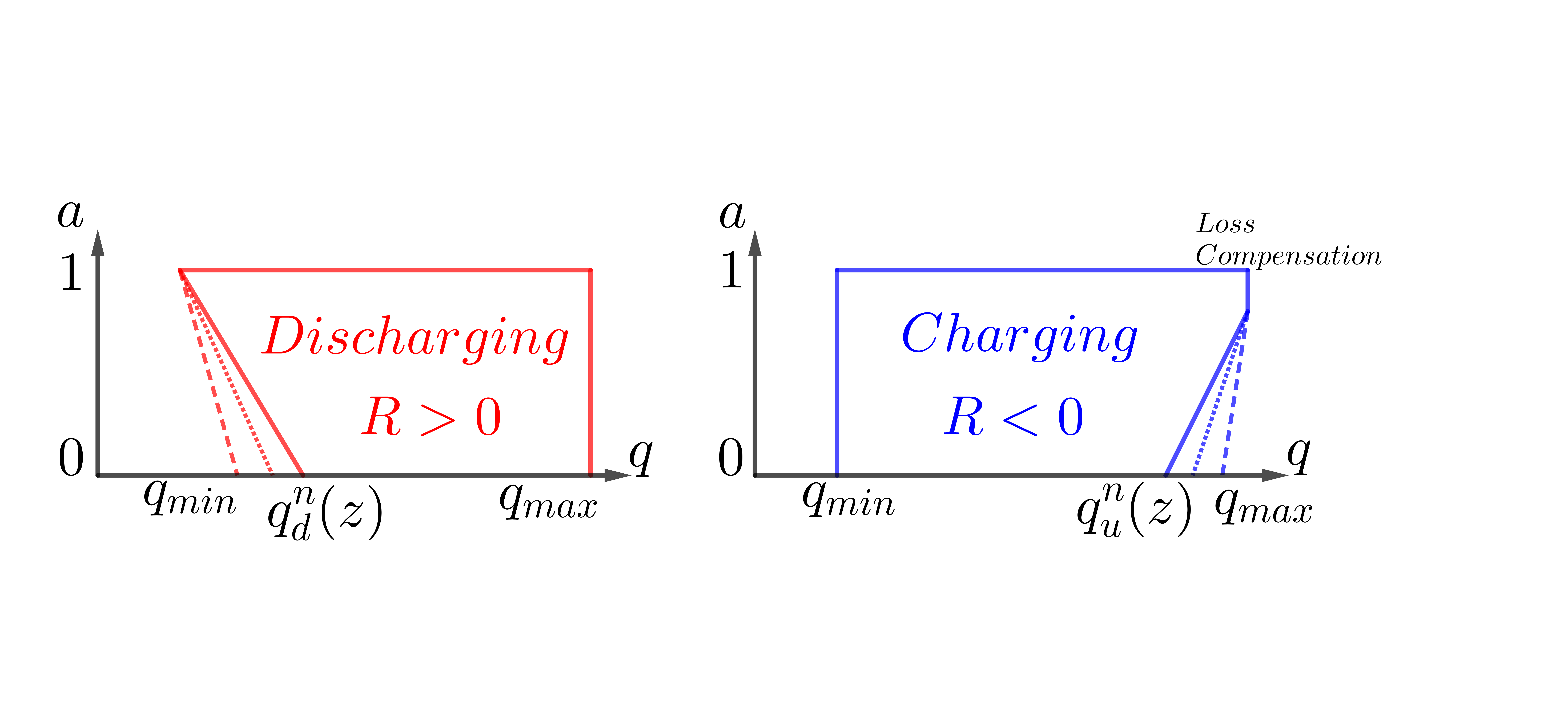}
				\vspace{-1.5cm}
				\caption{Set of feasible controls $\mathcal{K}^n_d$ with unsatisfied demand $z=z_1,z_{max},\frac{z_{max}}{2}$ (Left) for  and overproduction $z=z_2,z_{min},\frac{z_{min}}{2}$ (Right)  for a weak ES insultation at a time $t_n$.}
				\label{Nfeassets}
			\end{figure}
			
			At a time $t_n$, we observe that on the left side of Figure \ref{Nfeassets}, for $z=\frac{z_{max}}{2}$, $q^n_d$ is closer to $q_{min}$ as compared to $z=z_{max}$. Similarly at time $t_n$, we note that on the right side, $z=\frac{z_{min}}{2}$ moves $q^n_u$ closer to $q_{max}$ as compared to $z=z_{min}$. \newline
			
			In the following theorem, using the DPP, we derive the approximate optimal control and the terminal value problem. 
			\begin{theorem}
				\label{derivDS}
				Fix $(t_n,\zR,q)$ and assume Assumption \ref{AsumptrProsNoCN} holds.  Then from the DPP, we derive the approximation of the optimal control, $\alpha^{*n}_{\ell, j}$ given as
				\begin{equation}
					\label{doptcont}
					\alpha^{*n}_{\ell, j} = \argmin_{a\in \mathcal{K}^n_d(\zR,q)}\Big\{\int_{t_n}^{t_{n+1}}e^{-\delta (t-t_n)}\Psi(t,\zR_\ell,a)\mathrm{d}t+e^{-\delta \Delta t}V(t_{n+1},\zR,\Q^{a,n+1}_{j(\ell,n)})\Big\}.
				\end{equation}
				Moreover, setting $U(t,\zR) = V(t,\zR,\Q^{\alpha^*})$, we obtain the terminal value problem
				\begin{align}
					\label{dHJB}
					\frac{\partial U}{\partial t}(t,\zR) + \mathcal{L}U(t,\zR) + \Psi(r,\alpha^*) &= 0, \qquad \text{on~} [t_n,t_{n+1})\times \overline{\mathcal{X}} \\
					U(t_{n+1},\zR) &= V(t_{n+1},\zR,\Q^{\alpha^*,n+1}_{j(\ell,n)}), 
				\end{align}
				where $\mathcal{L}U=f(\zR)\frac{\partial U}{\partial \zR} + \frac{\sigmaR^2}{2}\frac{\partial^2 U}{\partial^2\zR^2}-\delta U$.
			\end{theorem}
			\begin{proof}
				See Appendix \ref{proofderivDS}.
			\end{proof}
			Let $\Lambda$ denote the discretization operator of the differential operator $\mathcal{L}$ and denote $\theta_{\ell} = -\kappaR \zR_{\ell}$. Applying the upwind discretization for the convection term $\frac{\partial V}{\partial \zR}$ and the central second-order finite difference for the diffusion term $\frac{\partial^2 V}{\partial \zR^2}$, we obtain  
			\[ \Lambda V^n_{\ell,j} = \frac{\sigmaR^2}{2}\frac{V^n_{\ell - 1,j} -2V^n_{\ell,j}+V^n_{\ell + 1,j}}{(\Delta \zR)^2}-\delta V^n_{\ell,j} + \left\{
			\begin{array}{ll}
				\theta_{\ell}\frac{V^n_{\ell,j} - V^n_{\ell-1,j}}{\Delta \zR}, & \theta_{\ell}\ge 0, \\\\
				\theta_{\ell}\frac{V^n_{\ell+1,j} - V^n_{\ell,j}}{\Delta \zR}, & \theta_{\ell}< 0,
			\end{array}
			\right. \]	
			\begin{equation}
				\label{opapprox}
				\hspace{-3cm}	= \mathcal{H}_{\ell}V^n_{\ell + 1,j}-\mathcal{F}_{\ell}V^n_{\ell,j}+\mathcal{D}_{\ell}V^n_{\ell - 1,j},
			\end{equation}
			where the expressions for $\mathcal{D}_{\ell}$, $\mathcal{F}_{\ell}$ and $\mathcal{H}_{\ell}$ are derived subsequently depending on the sign of $\theta_\ell$.\newline
			For $\theta_{\ell} \ge 0$, we obtain  
			\begin{equation}
				\label{positive_theta}
				\mathcal{D}_{\ell} = \frac{\sigmaR^2}{2(\Delta \zR)^2}-\frac{\theta_{\ell}}{\Delta \zR}, \quad \mathcal{H}_{\ell} = \frac{\sigmaR^2}{2(\Delta \zR)^2} \quad \text{and} \quad \mathcal{F}_{\ell} = \mathcal{D}_{\ell}+\mathcal{H}_{\ell}+\delta.
			\end{equation}
			Similarly for $\theta_{\ell} < 0$ we obtain
			\begin{equation}
				\label{negative_theta}
				\mathcal{D}_{\ell} = \frac{\sigmaR^2}{2(\Delta \zR)^2}, \quad \mathcal{H}_{\ell} = \frac{\sigmaR^2}{2(\Delta \zR)^2}+\frac{\theta_{\ell}}{\Delta \zR} \quad \text{and} \quad \mathcal{F}_{\ell} = \mathcal{D}_{\ell}+\mathcal{H}_{\ell}+\delta.
			\end{equation}
			where $\ell = 1,\ldots,N_{\zR}-1$ ~\text{and}~ $n = 0,\ldots,N_t-1$.\newline
			In the proposition below, we derive an upper bound for $\Delta \zR$, hence defining a positivity condition for the coefficients $\mathcal{D}_{\ell}, \mathcal{F}_{\ell}$ and $\mathcal{H}_{\ell}$.
			\begin{proposition}
				\label{positcond}
				For $\ell = 1,\ldots,N_{\zR}-1$ ~\text{and}~ $n = 0,\ldots,N_t-1$, the coefficients $\mathcal{D}_{\ell}, \mathcal{F}_{\ell}$ and $\mathcal{H}_{\ell}$ are positive if
				\begin{equation}
					\label{positivity_condition}
					\Delta \zR \le \frac{\sigmaR \sqrt{2\kappaR}}{6 \kappaR}.
				\end{equation}
			\end{proposition}
			\begin{proof}
				See Appendix \ref{proofpositcond}.
			\end{proof}
			Now, for simplicity, we require that the arrival point $\Q^{a,n+1}_{j(\ell,n)}$ always lies between $q_{j-1} = q_j-\Delta q$ and $q_{j+1} = q_j+\Delta q$ adjacent to $q_j$, hence leading to the following inequalities
			\begin{align}
				\label{inequalities}
				q_{min}:=q_0 &\le q_{min}-\Delta_1 \le q_1 := q_{min}+\Delta q, \qquad \Delta_1 \le 0, \nonumber\\
				q_{j-1} &\le q_j - \Delta_2 ~~~ \le q_{j+1}, \qquad j=1,\ldots,N_q-1, \\
				q_{max}-\Delta q:=q_{N_q-1} &\le q_{max} - \Delta_3 \le q_{N_q}:=q_{max}, \qquad \quad~~~ \Delta_3 \ge 0, \nonumber
			\end{align}
			where: 
			\begin{align*}
				\Delta_1 &= \frac{\Delta t}{m_Qc_P}(1-a)(\muR^n+\zR_\ell) + \frac{A\gamma\Delta t}{m_Qc_P}(q_{min}-Q_{amb}^n) \\
				\Delta_2 &= \frac{\Delta t}{m_Qc_P}(1-a)(\muR^n+\zR_\ell) + \frac{A\gamma\Delta t}{m_Qc_P}(q_j-Q_{amb}^n) \\
				\Delta_3 &= \frac{\Delta t}{m_Qc_P}(1-a)(\muR^n+\zR_\ell) + \frac{A\gamma\Delta t}{m_Qc_P}(q_{max}-Q_{amb}^n).
			\end{align*}
			In the lemma below, we formulate the Courant-Friedrichs-Lewy (CFL) (see \cite{courant1928partiellen}) condition which relates the time step $\Delta t$ to the step size $\Delta q$. It ensures that the arrival point $\Q^{a,n+1}_{j(\ell,n)}$ always lies between $q_{j-1} = q_j-\Delta q$ and $q_{j+1} = q_j+\Delta q$ adjacent to $q_j$.
			\begin{lemma}
				Let $\muR_{min} = \displaystyle \min_{t\in [0,T]}\muR(t)$, $\muR_{max} = \displaystyle \max_{t\in [0,T]}\muR(t)$ and denote $Q^{min}_{amb} = \displaystyle \min_{n\in \mathcal{N}}Q_{amb}^n$. If Equation \eqref{inequalities} is satisfied, then
				\begin{equation}
					\label{CFL}
					\Delta q \ge \frac{\Delta t}{m_Qc_P}(\max(|\muR_{min}+\zR_{min}|,\muR_{max}+\zR_{max})-A\gamma (q_{max}-Q^{min}_{amb})).
				\end{equation}
			\end{lemma}
			\begin{proof}
				The inequalities in Equation \eqref{inequalities} hold if $\Delta q\ge \frac{\Delta t}{m_Qc_P} |(1-a)(\muR^n+\zR_\ell) + A\gamma(q_j-Q_{amb}^n)|, \forall \ell,j$.\newline
				Now, we have $|(1-a)(\muR^n+\zR_\ell) + A\gamma(q_j-Q_{amb}^n)| \ge (1-a)|(\muR^n+\zR_\ell)|-A\gamma |q_j-Q_{amb}^n| \ge (1-a)|(\muR^n+\zR_\ell)|-A\gamma (q_{max}-Q_{amb}^n)~\forall \ell$.
			\end{proof} 
			Since $\Q^{a,n+1}_{j(\ell,n)}$ does not always coincide with a grid point $q_j\in \mathcal{G}_q$, this motivates an interpolation of $V(t_{n+1},\zR_\ell,\Q^{a,n+1}_{j(\ell,n)})$ by function values $V^{n+1}_{\ell,j}$ in the grid points of $\mathcal{G}$. According to \cite{chen2008semi}, it is enough to consider a linear interpolation since it allows to construct a monotone difference scheme. We denote by $V^{n+1}_{j(\ell,n)}$ the interpolation of $V(t_n,\zR_\ell,\Q^{a,n+1}_{j(\ell,n)})$ and construct it in the following proposition.\newline
			\begin{proposition}
				\label{lininterpol}
				Assume the CFL condition \eqref{CFL} and let \newline $b^n_{\ell,j}=\frac{\Delta t}{m_Qc_P\Delta q}\Big[(1-a)(\muR^n+\zR_\ell)+A\gamma(q_j-Q^{min}_{amb})\Big]$. Then the interpolated value $V^{n+1}_{j(\ell,n)}$ is represented by
				\begin{align}
					\label{approx}
					V^{n+1}_{j(\ell,n)} &= B^{(q),n}_{\ell,j}V^{n+1}_{\ell,j}+A^{(q),n}_{\ell,j}V^{n+1}_{\ell,j-1}+C^{(q),n}_{\ell,j}V^{n+1}_{\ell,j+1},
				\end{align} 
				for $\ell=0,\ldots,N_{\zR},~j=1,\ldots,N_{q}-1,~n=0,\ldots,N_t-1$ and where \newline $B^{(q),n}_{\ell,j} = 1-|b^n_{\ell,j}|, \quad A^{(q),n}_{\ell,j} = \frac{b^n_{\ell,j}+|b^n_{\ell,j}|}{2} \quad \text{~and~} \quad C^{(q),n}_{\ell,j} = -\frac{b^n_{\ell,j}-|b^n_{\ell,j}|}{2}$. \newline
				Furthermore, we obtain that for $j=0$ and $j=N_q$,
				\begin{align}
					V^{n+1}_{0(\ell,n)} &= B^{(q),n}_{\ell,0}V^{n+1}_{\ell,0}+C^{(q),n}_{\ell,0}V^{n+1}_{\ell,1}, \nonumber\\
					V^{n+1}_{N_q(\ell,n)} &= B^{(q),n}_{\ell,N_q}V^{n+1}_{\ell,N_q}+A^{(q),n}_{\ell,N_q}V^{n+1}_{\ell,N_q-1}. \nonumber
				\end{align}
			\end{proposition}
			\begin{proof}
				See Appendix \ref{prooflininterpol}.
			\end{proof}
			Discretizing Equation \eqref{dHJB} by a $\theta$-weighted, $\theta \in [0,1]$ implicit finite difference scheme yields
			\begin{equation}
				\frac{V^{n+1}_{\ell,j} - V^n_{\ell,j}}{\Delta t} + \theta \Lambda V^n_{\ell,j} + (1-\theta)\Lambda V^{n+1}_{\ell,j}+\Psi(\zR_\ell,\alpha^*) = 0.
			\end{equation}
			Taking $\theta = 1$, gives a fully implicit scheme
			\begin{align}
				\label{dscheme}
				V^n_{\ell,j} - \Delta t\Big\{\mathcal{H}_{\ell}V^n_{\ell + 1,j}-\mathcal{F}_{\ell}V^n_{\ell,j} + \mathcal{D}_{\ell}V^n_{\ell - 1,j}\Big\} &= V^{n+1}_{j(\ell,n)} + \Delta t\Psi(\zR_{\ell},\alpha^*), \nonumber\\
				\Q^{a,n+1}_{j(\ell,n)}&= \Big(1-\frac{A\gamma \Delta t}{m_Qc_P}\Big)q_j + \frac{\Delta t}{m_Qc_P}((a-1)(\muR^n+\zR_\ell) \nonumber \\
				&\hspace{1cm} + A\gamma Q_{amb}^n), \\
				\alpha^{*n}_{\ell, j} &= \argmin_{a\in \mathcal{K}^n_d(\zR_{\ell},q_j)}\Big\{\int_{t_n}^{t_{n+1}}e^{-\delta (t-t_n)}\Psi(t,\zR_\ell,a)\mathrm{d}t \nonumber \\ 
				&\hspace{1cm} + e^{-\delta \Delta t}V^{n+1}_{j(\ell,n)}\Big\} \nonumber\\
				\mathcal{K}^n_d(\zR_{\ell},q_j) &= \{a\in \mathcal{K}_a(t,\zR_{\ell},q_j) \quad |~ \Q^{a,n+1}_{j(\ell,n)}\in [q_{min},q_{max}]\} \nonumber\\
				V^{N_t}_{\ell,j} &= \Phi(q_j). \nonumber	
			\end{align}
			Set $\Psi^{n+1}_{\ell,j} = \Psi(\zR_\ell,a)$. Let $n$ and $j$ be fixed, and let the optimal strategy $a=\alpha_{\ell, j}^{*n}$ together with the expressions \eqref{CFL} - \eqref{approx} $\forall \ell$ be known.\newline
			Denote the right hand side of the difference equation in \eqref{dscheme} by $\Gamma_{\ell,j}^{n+1}$. We have that $\Gamma_{\ell,j}^{n+1}$ is known for fixed $t_n$ and $q_j$ and $\forall \zR_\ell$.\newline
			Thus we obtain 
			\begin{equation}
				\label{systemequations}
				(1+\Delta_{N_t} \mathcal{F}_\ell)V^n_{\ell,j}-\Delta_{N_t} \mathcal{D}_\ell V^n_{\ell-1,j}- \Delta_{N_t} \mathcal{H}_\ell V^n_{\ell+1,j} = \Gamma^{n+1}_{\ell,j},
			\end{equation}
			where $\Gamma_{\ell,j}^{n+1} = A^{(q),n}_{\ell,j}V^{n+1}_{\ell,j-1}+B^{(q),n}_{\ell,j}V^{n+1}_{\ell,j}+C^{(q),n}_{\ell,j}V^{n+1}_{\ell,j+1}+\Delta_{N_t} \Psi^{n+1}_{\ell,j}$ for $\ell=1,\ldots,N_{\zR}-1,~j=1,\ldots,N_q-1,~n=0,\ldots,N_t-1$.\newline
			Equation \eqref{systemequations} is a system of linear algebraic equations and is useful in the sequel to obtain $V^n_{\ell,j}$ in the interior points.\newline
			Note that at the boundary $q=q_{min},q_{max}$, we do not impose ``stringent" boundary conditions. Hence in the following, we derive conditions at the boundaries $q_{min}$ and $q_{max}$ for which PDE \eqref{dHJB} can be solved numerically with no further requirement in the $q$-direction as in \cite{chen2008semi}. At $q=q_{min}$, we have from the first equation of \eqref{inequalities} that $\Delta_1 \le 0$. Hence substituting $j=0$, for $q=q_{min}$ in $b_{\ell,j}$, it implies that
			\begin{equation}
				\label{q1BC}
				b^n_{\ell,0} = \frac{\Delta t}{m_Qc_P\Delta q}((1-a)(\muR^n+\zR_\ell) + A\gamma (q_{min}-Q_{amb}^n))\le 0.
			\end{equation}
			Similarly at $q=q_{max}$, from the last Equation in \eqref{inequalities}, we have $\Delta_3\ge 0$. This further implies that for $j=N_q+1$,
			\begin{equation}
				\label{q2BC}
				b^n_{\ell,N_q} = \frac{\Delta t}{m_Qc_P\Delta q}((1-a)(\muR^n+\zR_\ell) + A\gamma (q_{max}-Q_{amb}^n))\ge 0.
			\end{equation}
			The signs of $b^n_{\ell,0}$ and $b^n_{\ell,N_q}$ justify the use
			Therefore we solve the PDE \eqref{dHJB} in the $q$-direction with no further requirement except conditions \eqref{q1BC} and \eqref{q2BC}.
			\subsection{Boundary conditions}
			In this section, we formulate boundary conditions at $\zR=\zR_{min}, \zR_{max}$ which are needed due to the truncation of PDE domain from $\mathcal{X}$ to $\overline{\mathcal{X}}$. In order to reduce the effects due to mispecification of the boundary conditions, we proceed as in the literature (see. \cite{chen2008semi,shardin2017partially}). We require that the second order derivative of $V$ vanishes at the boundaries $\zR_{min}$ and $\zR_{max}$, while no further requirements must be made to the value function and its first order derivative. Mathematically, we require that
			\begin{equation}
				\label{rboundcond}
				\frac{\partial^2 V}{\partial \zR^2}(t,\zR_{min},q) = 0, \quad \frac{\partial^2 V}{\partial \zR^2}(t,\zR_{max},q) = 0, \quad (t,q)\in (0,T)\times [q_{min},q_{max}].
			\end{equation}
			\subsection{Matrix formulation}
			\label{matrix formulation}
			From  the boundary conditions defined above together with $\Gamma_{\ell,j}^{n+1}$ known for $n$ and $j$ fixed, we formulate the following system of linear equations for the difference Equation \eqref{dscheme}
			\begin{equation}
				\label{linsys}
				\mathcal{M}\bm{V}^n_j = \bm{\Gamma}_j^{n+1},\quad \text{for~} n=0,\ldots,N_t -1, \quad j=1,\ldots,N_q -1,
			\end{equation}
			where $\bm{V}^n_j = (V^n_{1,j},\ldots,V^n_{N_{\zR}-1,j})^T$, \quad $\bm{\Gamma}_j^{n+1} = (\Gamma_{1,j}^{n+1},\ldots,\Gamma_{N_{\zR}-1,j}^{n+1})^T$ 
			with $\mathcal{M}$ an $(N_{\zR}-1)\times (N_{\zR}-1)$ tridiagonal matrix given as
			\begin{equation}
				\label{matrixM}
				\resizebox{.5\hsize}{!}
				{$\mathcal{M} = 
					\begin{bmatrix}
						a_1 & c_1 & 0 & 0 & \dots & 0 & 0\\
						b_2 & a_2 & c_2 & 0 & \dots & 0 & 0 \\
						0 & b_3 & a_3 & c_3 & \dots & 0 & 0 \\
						\dots & \dots & \dots & \dots & \dots & \dots & \dots \\
						0 & 0 & 0 & \dots & b_{N_{\zR}-2} & a_{N_{\zR}-2} & c_{N_{\zR}-2} \\
						0 & 0 & 0 & \dots & 0 & b_{N_{\zR}-1} & a_{N_{\zR}-1} 
					\end{bmatrix}
					$}
			\end{equation}
			where:
			\begin{align*}
				a_1 &= 1+\Delta t \mathcal{F}_1-2\Delta t \mathcal{D}_1, \qquad \qquad~~~ c_1 = \Delta t(\mathcal{D}_1-\mathcal{H}_1), \\
				a_j &= 1+\Delta t \mathcal{F}_j, \quad b_j = -\Delta t \mathcal{D}_j, \qquad~ c_j = -\Delta t \mathcal{H}_j, \quad j = 2,\cdots,N_{\zR}-1 \\
				a_{N_{\zR}-1} &= 1+\Delta t \mathcal{F}_{N_{\zR}-1}-2\Delta t \mathcal{H}_{N_{\zR}-1} \quad b_{N_{\zR}-1} = -\Delta t(\mathcal{D}_{N_{\zR}-1}-\mathcal{H}_{N_{\zR}-1}). 	
			\end{align*}
			In the follwoing, we derive an approximation of boundary values $V^n_{0,j}, V^n_{N_{\zR},j}, V^n_{l,0}$ and $V^n_{l,N_q}$. Substituting the boundary condition, Equation \eqref{rboundcond} in the first expression of Equation \eqref{dHJB} yields
			\begin{align}
				\label{BC1}
				\frac{\partial V}{\partial t}(t,\zR,q) + \mathcal{L}_1V(t,\zR,q) + \Psi(\zR,\alpha^*) &= 0,
			\end{align}
			where $\mathcal{L}_1V = -\kappaR \zR_{min}\frac{\partial V}{\partial \zR}$. Taking $\ell=0$, we observe by construction that $\theta_0 = -\kappaR \zR_{min}\ge 0$. Hence, forward discretization can be used to approximate $\mathcal{L}_1$. Let $\Lambda_1$ denote this approximation given by
			\begin{equation}
				\label{dBC1}
				\Lambda_1 V^n_{0,j} = \mathcal{H}_0V^n_{1,j} - \mathcal{F}_0V^n_{0,j}, 
			\end{equation}
			where $\mathcal{H}_0 = \frac{\theta_0}{\Delta \zR}$ and $\mathcal{F}_0 = \frac{\theta_0}{\Delta \zR} + \delta$, for $n = 0\cdots,N_t-1$, and $j = 1\cdots,N_q-1$.\newline
			Next substituting Equation \eqref{dBC1} into the first Equation in \eqref{dscheme} gives
			\begin{equation}
				\label{BC11}
				(1+\Delta_{N_t} \mathcal{F}_0)V^n_{0,j} = V^{n+1}_{j(0,n)} + \Delta_{N_t}(\mathcal{H}_0V^n_{1,j} + \Psi(\zR_{min},\alpha^*)).
			\end{equation}
			Recall that $\Gamma_{0,j}^{n+1}$ is known for $n$ and $j$ fixed. Thus from Equation \eqref{BC11}, we derive the following system of linear equations corresponding to the boundary $\zR = \zR_{min}$
			\begin{equation}
				\label{linsys2}
				\mathcal{B}\bm{V}^n_0 = \bm{\Gamma}_0^{n+1} + \Delta_{N_t} \mathcal{H}_0\bm{V}^n_1,\quad \text{for~} n=0,\ldots,N_t -1,
			\end{equation}
			$\bm{V}^n_0 = (V^n_{0,1},\ldots,V^n_{0,N_q-1})^T$, \quad $\bm{\Gamma}_0^{n+1} = (\Gamma_{0,1}^{n+1},\ldots,\Gamma_{0,N_q-1}^{n+1})^T$, \quad $\bm{V}^n_1 = (V^n_{1,1},\ldots,V^n_{1,N_q-1})^T$ 
			with $\mathcal{B}$ an $(N_{\zR}-1)\times (N_{\zR}-1)$ diagonal matrix (see Appendix \ref{MatrixBC}).\newline
			For $\ell=N_{\zR}$, we observe that by construction, $\theta_{N_{\zR}} = -\kappaR \zR_{max}< 0$. Let $\Lambda_2$ denote the approximation of $\mathcal{L}_1$ when using backward discretization. Thus $\Lambda_2$ is given by
			\begin{equation}
				\label{dBC2}
				\Lambda_2 V^n_{N_{\zR},j} = A^n_{N_{\zR}}V^n_{N_{\zR}-1,j} - B^n_{N_{\zR}}V^n_{N_{\zR},j}, 
			\end{equation}
			where $\mathcal{D}_{N_{\zR}} = -\frac{\theta_{N_{\zR}}}{\Delta \zR}$ and $\mathcal{F}_{N_{\zR}} = \delta - \frac{\theta_{N_{\zR}}}{\Delta \zR}$, for $n = 0\cdots,N_t-1$, and $j = 1\cdots,N_q-1$.\newline
			Substituting Equation \eqref{dBC2} into the first Equation in \eqref{dscheme} gives
			\begin{equation}
				\label{BC12}
				(1+\Delta t \mathcal{F}_{N_{\zR}})V^n_{N_{\zR},j} = V^{n+1}_{j(N_{\zR},n)} + \Delta t(\mathcal{D}_{N_{\zR}}V^n_{N_{\zR}-1,j} + \Psi(\zR_{max},\alpha^*)).
			\end{equation}
			Note that $\Gamma_{N_{\zR},j}^{n+1}$ is known for $n$ and $j$ fixed. Thus from Equation \eqref{BC12}, we derive the following system of linear equations corresponding to the boundary $\zR = \zR_{max}$
			\begin{equation}
				\label{linsys3}
				\mathcal{C}\bm{V}^n_{N_{\zR}} = \bm{\Gamma}_{N_{\zR}} + \Delta t \mathcal{D}_{N_{\zR}}\bm{V}^n_{N_{\zR}-1},\quad \text{for~} n=0,\ldots,N_t -1,
			\end{equation}
			$\bm{V}^n_{N_{\zR}} \quad = (V^n_{N_{\zR},1},\ldots,V^n_{N_{\zR},N_q-1})^T$, \quad $\bm{\Gamma}_{N_{\zR}}^{n+1} = (\Gamma_{N_{\zR},1}^{n+1},\ldots,\Gamma_{N_{\zR},N_q-1}^{n+1})^T$, \newline \quad $\bm{V}^n_{N_{\zR}-1} = (V^n_{N_{\zR}-1,1},\ldots,V^n_{N_{\zR}-1,N_q-1})^T$ 
			with $\mathcal{C}$ an $(N_{\zR}-1)\times (N_{\zR}-1)$ diagonal matrix (see Appendix \ref{MatrixBC}).\newline
			Now, substituting $j = 0$ in the first expression in Equation \eqref{dscheme} we get
			\begin{equation}
				\label{BC21}
				V^n_{\ell,0} - \Delta t\Big\{\mathcal{H}_{\ell}V^n_{\ell + 1,0}-\mathcal{F}_{\ell}V^n_{\ell,0} + \mathcal{D}_{\ell}V^n_{\ell - 1,0}\Big\} = V^{n+1}_{0(\ell,n)} + \Delta t\Psi(\zR_{\ell},\alpha^*)
			\end{equation}
			Also note that $\Gamma_{\ell,0}^{n+1}$ is known for $n$ and $\ell$ fixed. Varying the values of $\ell = 1,\cdots,N_{\zR}-1$, we derive the following system of linear equations for the boundary $q = q_{min}$
			\begin{equation}
				\label{linsys4}
				\mathcal{M}\bm{V}^n_0 = \bm{\Gamma}_0^{n+1},\quad \text{for~} n=0,\ldots,N_t -1, \quad \ell=1,\ldots,N_{\zR} -1,
			\end{equation}
			where $\bm{V}^n_0 = (V^n_{1,0},\ldots,V^n_{N_{\zR}-1,0})^T$, \quad $\bm{\Gamma}_0^{n+1} = (\Gamma_{1,0}^{n+1},\ldots,\Gamma_{N_{\zR}-1,0}^{n+1})^T$ with the matrix $\mathcal{M}$ defined in Equation \eqref{matrixM}.\newline
			Similarly, substituting $j = N_q$ in the first expression in Equation \eqref{dscheme} we obtain
			\begin{equation}
				\label{BC22}
				V^n_{\ell,N_q} - \Delta t\Big\{\mathcal{H}_{\ell}V^n_{\ell + 1,N_q}-\mathcal{F}_{\ell}V^n_{\ell,N_q} + \mathcal{D}_{\ell}V^n_{\ell - 1,N_q}\Big\} = V^{n+1}_{N_q(\ell,n)} + \Delta t\Psi(\zR_{\ell},\alpha^*)
			\end{equation}
			Varying the values of $\ell = 1,\cdots,N_{\zR}-1$ and noting that $\Gamma_{\ell,N_q}^{n+1}$ is known for $n$ and $\ell$ fixed, we also derive the following system of linear equations for the boundary $q = q_{max}$
			\begin{equation}
				\label{linsys5}
				\mathcal{M}\bm{V}^n_{N_q} = \bm{\Gamma}_{N_q}^{n+1},\quad \text{for~} n=0,\ldots,N_t -1, \quad \ell=1,\ldots,N_{\zR} -1,
			\end{equation}
			where $\bm{V}^n_{N_q} = (V^n_{1,N_q},\ldots,V^n_{N_{\zR}-1,N_q})^T$, \quad $\bm{\Gamma}_{N_q}^{n+1} = (\Gamma_{1,N_q}^{n+1},\ldots,\Gamma_{N_{\zR}-1,N_q}^{n+1})^T$ with the matrix $\mathcal{M}$ defined in Equation \eqref{matrixM}.\newline
			Finally in order to determine the corner values $V^n_{0,0}, V^n_{0,N_q}, V^n_{N_{\zR},0}$ and $V^n_{N_{\zR},N_q}$, we use previously computed values $V^n_{\ell,j}$ for $\ell = 0,\cdots,N_{\zR}$, $j = 0,\cdots,N_q$ and obtain
			\begin{align*}
				\label{conner}
				V^n_{0,0} &= 2V^n_{1,0} - V^n_{2,0}, \\
				V^n_{0,N_q} &= 2V^n_{0,N_q-1} - V^n_{0,N_q-2}, \\
				V^n_{N_{\zR},0} &= 2V^n_{N_{\zR},1} - V^n_{N_{\zR},2}, \\
				V^n_{N_{\zR},N_q} &= 2V^n_{N_{\zR},N_q-1} - V^n_{N_{\zR},N_q-2}.
			\end{align*}
			The approximate optimal control in Equation \eqref{doptcont} and the value functions from the system \eqref{linsys} can be solved by using Algorithm 1, starting at the terminal time $N_t$. 
			\begin{table}[!h]
				\label{algorithm1}
				\begin{center}
					\begin{tabular}{l} 
						\hline
						\textbf{Algorithm}: Backward recursion algorithm \\
						\hline
						\textbf{Result:} Find the value function $V$ and the optimal strategy $\alpha^{*n}_{\ell, j}$ \\
						\vspace{.25cm}
						\textbf{Step 1} Compute for all $(\zR_\ell,q_j)\in \mathcal{G}_{\zR} \times \mathcal{G}_q$\\
						\hspace{5cm} $V(N_t,\zR_\ell,q_j) = \Phi(q_j)$ \\
						\vspace{.25cm}
						\textbf{Step 2} For $n:=N_t-1,\ldots,1,0$ compute for all $(\zR,q)\in \mathcal{S}\times [q_{min},q_{max}]$ \\
						\hspace{2.5cm} $\alpha^{*n}_{\ell, j} = \displaystyle\argmin_{a\in \mathcal{K}^n_d(\zR_\ell,q_j)}\Big\{\int_{t_n}^{t_{n+1}} e^{-\delta (s-t)}\Psi(t,\zR_\ell,a)
						\mathrm{d}s+e^{-\delta \Delta t}V^{n+1}_{j(\ell,n)}\Big\}.$ \\
						\vspace{.25cm}
						Compute the following value function $V(n,\zR_\ell,q_j)$ as in Theorem \ref{derivDS} \\
						\hspace{2.5cm} $V(n,\zR_\ell,q_j)=\mathbb{E}_{t_n,\zR_\ell,q_j}\Big\{\int_{t_n}^{t_{n+1}} e^{-\delta (s-t)}\Psi(t,\ZR(s),\alpha^{*n}_{\ell, j})
						\mathrm{d}s+e^{-\delta \Delta t}V^{n+1}_{j(\ell,n)}\Big\}.$ \\
						\hline
					\end{tabular}
				\end{center}
			\end{table}
			In order to solve for $\alpha_{\ell, j}^{*n}$, we first distinguish cases of unsatisfied demand and overproduction since they inform us about which subinterval of $\mathcal{K}^n_d$ is to be considered. Then the optimization step is implented in order to return the control with the smallest/minimum objective function. This control then represents the optimal control given specific
			$\zR$ and $q$. The above step is repeated until all optimal controls are obtained.
			
			\section{Numerical results}
			\label{numericalresults}
			In this section we discuss numerical results of the prosumer's optimal control problem. We implement the difference scheme we derived using the DPP based on a time and state discretization. The simulations based on the above backward recursion algorithm aim to primarily determine the value function and the optimal charging and discharging strategies of the ES at any discrete time points and at any grid point in the discretized state-space $\overline{\mathcal{X}}$, and subsequently gain meaningful insights on the properties of the value function. We present the results considering the case of a liquidation problem. That is
			\begin{equation}
				\label{liquidation}
				\Phi(q) = -S_{liq}m_Qc_P(q-q_{min}),
			\end{equation}
			with the liquidation price satisfying the condition $S_{liq} < S_{sell}(T), \quad T>0$.\newline
			For the purpose of the numerical simulations, we consider the seasonality function
			\begin{equation}
				\label{seasfunc1}
				\muR(t)=L_{0}+ L\cos\Big(\frac{2\pi}{\rho}(t-\widetilde{t})\Big), 
			\end{equation}
			with $\rho = 365\times24~h$ and $\widetilde{t} = 0$. The parameters $L_0$ and $L$ are calibrated to the model in such a way that starting with a full storage at initial time, the prosumer can drive ES empty at a given time after satisfying the residual demand $L_0+L$. The parameter $\gamma$ is calibrated such that starting with a full storage and assuming no charging or discharging, ES self discharges and reaches a certain level at a specified time.\newline
			In the case of the yearly heat price, we consider for the purpose of the numerical simulations, $S_{buy}$ and $S_{sell}$ defined as
			\begin{equation}
				\label{SbuySsell}
				S_{buy}(t)=L^S_0+ L^S\cos\Big(\frac{2\pi}{\rho_S}(t-t_S)\Big), \quad S_{sell}(t) = S_{buy}(t) - \xi,
			\end{equation}
			and choose the values $L_0^S = 0.17$~$[\EUR/kWh]$, $L^S = 0.15$~$[\EUR/kWh]$ $\xi = 0.02$~$[\EUR/kWh]$, $\rho_S = 365\times 24~h$ and $t_S=0$. Equation \ref{SbuySsell} reflects the fact that heat is more valuable in cold seasons than it is in warm seasons. Hence $S_{buy}$ and $S_{sell}$ become higher for a higher seasonality function $\muR$ and decrease with respect to $\muR$ as well.\newline
			Recall that for $\alpha \in [0,1]$, $\alpha = 0$ implies that the prosumer satisfies all its residual demand from ES, whereas $\alpha = 1$ means that the prosumer satisfies all its residual demand from CHS. In the following, we consider as ES a cylinder of radius $r_d = 1~m$, height $h_0 = 2.5~m$ resulting in $V=7.85~m^3$ and surface area $A=21.99~m^2$. The mass of water in ES is $m_Q=7854 ~kg$ with specific heat capacity $c_P=0.0012~\frac{kWh}{kg K}$. We assume that minimum and maximum temperature in ES are $q_{min}=25 ^{\circ}C$ and $q_{max}=85 ^{\circ}C$ respectively and derive the maximum amount energy that can be stored in ES as $\mathcal{E}=m_Qc_P(q_{max}-q_{min})=547.94~kWh$. \newline
			For the deseasonalized residual demand $\ZR$, using the $3\sigma$-rule, we obtain $\zR_{min} = -2~kW$, $\zR_{max} = 2~kW$.\newline
			We consider as time horizon $T = 365\times 24~h = 8760~h$ and time step size $\Delta t = 1~h$.
			The full description of the model parameters can be found in Table \ref{tab:1}.\newline
			In the following, we let $Q^n_{amb} = q_{min}$ for simplicity and consider 3 cases; the basic, the strong seasonality and the weak insulation cases to illustrate our numerical results. Let $L^{Basic}_0, L^{Basic}$ denote respectively, the long term residual demand and amplitude of seasonality for the basic case. $L^{Strong}_0, L^{Strong}$ denote respectively, the long term residual demand and amplitude of seasonality for the strong seasonality case. We also discuss the limiting case of perfect insulation, i.e. $\gamma^{Perfect} = 0$ in order to compare the value functions of all cases.
			\begin{table}[!h]
				\begin{center}
					\begin{tabular}{ |c|c||c|c| } 
						\hline
						Parameters & Numerical values & Parameters & Numerical values \\
						\hline 
						$L^{Basic}_0$ & 0.37 & $L^{Strong}_0$ & 0.37\\ [.5pt]
						$L^{Basic}$ & 1.00 & $L^{Strong}$ & 4.04\\ [5pt] 
						\hline
					\end{tabular}
					\caption[Parameter values of the long term residual demand and amplitude in $kW$.]{Parameter values of the long term residual demand and amplitude in $kW$.}
					\label{tab:3}
				\end{center}
			\end{table}
			The calibration of the parameters in Table \ref{tab:3} is provided in Section \ref{calibrationL_0L}.
			
			\subsection{Optimal strategy and value function}
			\paragraph{Terminal value function}
			We consider as terminal condition the liquidation problem at terminal time $T=8760~h$. From the formulation of the liquidation problem in Equation \eqref{liquidation}, we expect the terminal value function to be negative meaning that the prosumer sells all excess production to CHS for a revenue. In Figure \ref{TermVal}, we observe that the revenue is higher as the leftover thermal energy in ES increase. In addition, the liquidation problem suggests that the terminal condition is independent of the deseasonalized residual demand. This motivates the fact that the value function is constant in the residual demand. 
			\begin{figure}[ht!]			
				\centering
				\includegraphics[width=.75\textwidth]{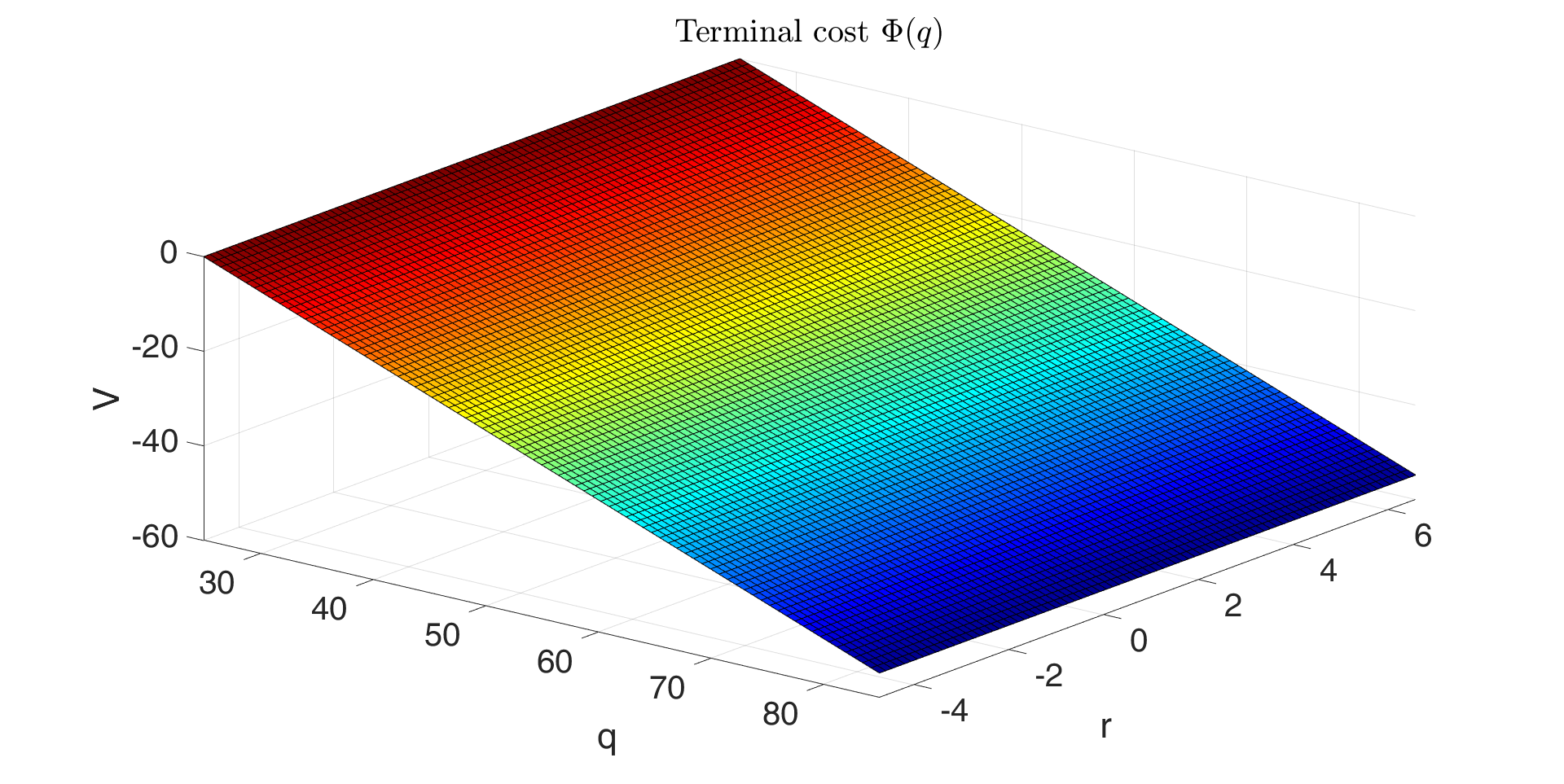}
				\caption[Terminal cost $\Phi$ with a liquidation problem.]{Terminal cost $\Phi$ with a liquidation problem.}
				\label{TermVal}
			\end{figure}
			
			Throughout the numerical results, our goal is to understand the impact of the seasonality and ES insulation on the model. To investigate the impact of the insulation, we present and discuss the results for the basic and weak insulation cases. In each of the 3 cases we discuss below, we show the optimal strategy for different time points in order to gain more insights. We also present the value function and optimal strategy at $t=0$ as functions of the ES temperature level and residual demand and discuss the results. For this, we show the value function on the left and the optimal strategy on the right. We note that our optimal control is of bang-bang type with the values changing according to the boundaries of the subintervals in $\mathcal{K}^n_d$.
			In the following, we present the numerical results only for 3 cases and subsequently report the values of the value functions on an appropriate table for further discussions. For a simplicity in the discussion of the optimal control for different time points, we illustrate the results for $\hat{\alpha}(t) = \alpha^*(t)sign(r)$, where $sign(r)$  is the sign function of the residual demand $r$ and $\alpha^*(t)~t\in[0,T]$. Depending on the sign of $r$, $\hat{\alpha}(t)$ only takes the values $-1,0,1$. For $\hat{\alpha}(t) = 1$, the prosumer purchases all unsatisfied demand from the CHS, whereas $\hat{\alpha}(t) = -1$ implies that the prosumer sells all overproduction to the CHS. $\hat{\alpha}(t) = 0$ means that the prosumer uses only the ES to satisfy its residual demand. Hence, in the case $r>0$, $\hat{\alpha}(t) = 0$ implies that all unsatisfied demand is satisfied from the ES while for $r<0$, $\hat{\alpha}(t) = 0$ means that all overproduction is stored in the ES.
			
			\paragraph{Case I: Value function and optimal strategy for the basic case}
			\label{caseI}
			In our context, the basic case is the reference case to which both the weak insulation and strong seasonality cases are compared to in order to understand the impact of seasonality and storage insulation to the model. Let $\gamma^{Basic}$ denote the heat transfert coefficient for the basic case. From the calibration (see Section \ref{calibration_gamma}), we obtain $\gamma^{Basic} = 2.34\times 10^{-4}~\frac{kW}{m^2 K}$. The parameters $L^{Basic}_0,~ L^{Basic}$ of the seasonality function are provided in Table \ref{tab:3}.\newline
			Figure \ref{NZG1Valfunct1BasicNew} shows results of the optimal strategy for $t = 0,~120~,250,~300$ days. In our model, for a non-constant seasonality, $t=0,~300~days$ correspond to periods dominated by cold weather leading to more unsatisfied demands. In contrast, $t=120,~250~days$ correspond to periods dominated by warm weather leading to more overproduction. For $t=0$, in the event of an unsatisfied demand, the prosumer buys all heat from CHS when ES is empty; otherwise all residual demand is satisfied by discharging ES. However, in the event of an overproduction, the prosumer mostly sells its overproduction to the CHS for a revenue. At $t=120$ days, in the event of an overproduction, we observe that the prosumer mostly sells all overproduction to CHS. Now, at $t=250$ days, for an unsatisfied demand, the prosumer mostly purchases all heat from CHS though ES in non-empty. This is because $S_{buy}$ is low and more unsatisfied demands are expected in the future for higher $S_{buy}$. Since the ES has a good insulation, the prosumer postpones the use of the ES in order to save the thermal energy in order to satisfy the higher unsatisfied demands ahead which will otherwise require a higher cost due to the higher future heat buying price, $S_{buy}$.  In the event of an overproduction, the prosumer charges the ES  and sells all excess production for $q=q_{max}$. Finally, for $t=300$, when faced with an overproduction, the prosumer mostly charges ES until it is sufficiently charged and then sells the excess production to CHS for a revenue. This is again because more unsatisfied demands are expected for even higher $S_{buy}$; and thanks to the good insulation, the ES can later be discharged for a lesser cost.
			\begin{figure}[ht!]			
				\hspace{-1.25cm}
				\includegraphics[width=1.15\textwidth]{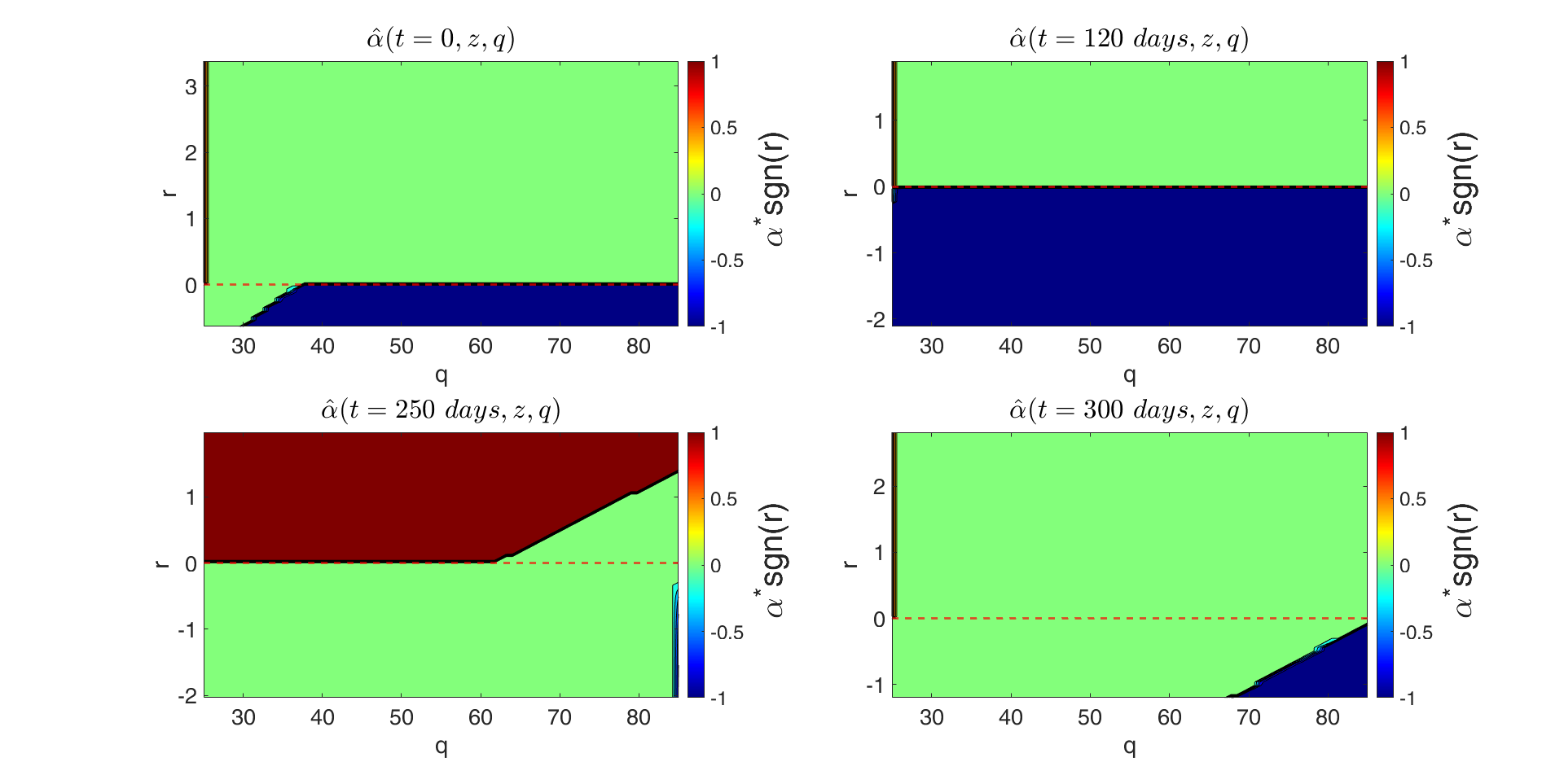}
				\caption[Optimal strategies at different time values for the basic case.]{Optimal strategy at $t = 0$ (Top left), $t = 120$ days (Top right), $t = 250$ days, (Bottom left), $t = 300$ days (Bottom right) for the basic case.}
				\label{NZG1Valfunct1BasicNew}
			\end{figure}
			\begin{figure}[ht!]			
				\hspace{-1.5cm}
				\includegraphics[width=1.15\textwidth]{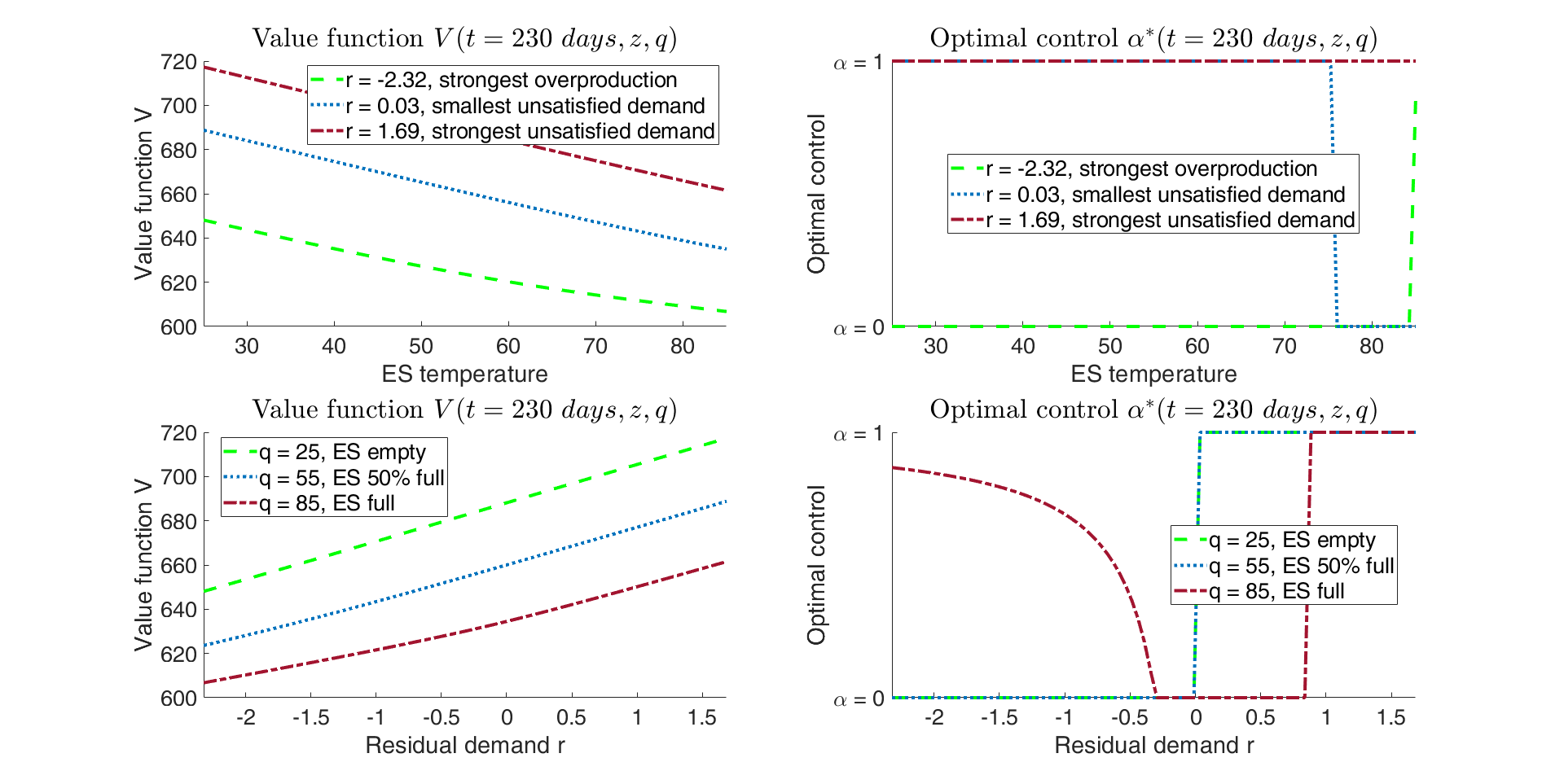}
				\caption[Value function and optimal strategy as functions of $Q^\alpha$ for different values of $R$ at $t=230~$days for the basic case.]{Value function (left) and optimal strategy (right) as functions of ES temperature $q$ and residual demand $r$ at $t=230~$days for the basic case.}
				\label{NZG1Valfunct2BasicNew}
			\end{figure}
			
			In Figure \ref{NZG1Valfunct2BasicNew}, we study the behavior of the value function and optimal strategy with respect to ES temperature $q$ and the residual demand $r$at time $t=230~$days. For the ES temperature level, we consider $q = 25,~55,~85^{\circ}\text{C}$ and $r = -2.32,~0.03,~1.69~kW$ for the residual demand.\newline
			The top left panel shows that for $r = 0.03,~1.69~kW$, the expected aggregated discounted cost (value function) of satisfying hot water and heating demand reduces as more temperature in the ES is available. Furthermore, we note that for a strongest unsatisfied demand, $r = 1.69~kW$, the prosumer incurs the highest expected aggregated discounted cost as compared to when faced with the smallest unsatisfied demand. In the top righ panel, for $r = 1.69~kW$, the prosumer purchases all the thermal energy despite the ES being sufficiently charged. This is due to the fact that for a good ES insulation, the prosumer postpones discharging the ES to satisfy the current unsatisfied demand since stronger unsatisfied demands are expected in the future for higher heat buying price $S_{buy}$. However for the smallest residual demand $r = 0.03~kW$, it discharges the ES. For the strongest overproduction, the prosumer charges the ES and only sells to the CHS when ES is full. \newline
			The bottom left panel shows that for a given ES temperature level, the value function increases with respect to the residual demand. However it is smaller when ES is fully charged as compared to an empty ES. In the bottom right panel, for $q=25^{\circ}\text{C},~55^{\circ}\text{C}$, the prosumer charges the ES in the event of an overproduction and discharges the ES otherwise. For a full ES, i.e. $q= q_{max}$, the prosumer compensates for the loss in the case of overproduction. For small unsatisfied demands, the prosumer discharges the ES and for stronger unsatisfied demands, it purchases thermal energy from the CHS.
			
			\paragraph{Case II: Value function and optimal strategy for a strong seasonality}
			In this parargaph we show the value function and optimal strategy for a strong seasonality. Further we study the behavior of the value function and optimal strategy with respect to the ES temperature $q$ and the residual demand $r$. For the illustrations below, we consider the values of the parameters $L^{Strong}_0,~L^{Strong}$ provided in Table \ref{tab:3}. We assume the same insulation as in the basic case, hence allowing to focus on the impact of the seasonality on the model. In Figure \ref{NZG1Valfunct1StrongNew}, we consider the time points $t=0,~120,~250,~300~$days. First, we note that the time points $t=0,~300~$days correspond to cold periods while $t=120,~250~$days correspond mostly to warm periods. In both the top left and bottom right panels, for an empty ES, the prosumer purchases thermal energy from the CHS, otherwise discharges ES. In the top right panel characterized mostly by overproduction, the prosumer sells all excess production to the CHS for a revenue, whereas in the down left panel it stores the overproduction if $q<q_{max}$. For $q = q_{max}$, the prosumer also sells the overproduction to the CHS for a revenue. In addition, when faced wih an unsatisfied demand in the top right panel, the prosumer purchases thermal energy from CHS if $q = q_{min}$ and otherwise discharges the ES. In contrast, in the bottom left panel, most unsatisfied demand is satisfied from the CHS.  
			
			\begin{figure}[ht!]			
				\hspace{-1.25cm}
				\includegraphics[width=1.15\textwidth]{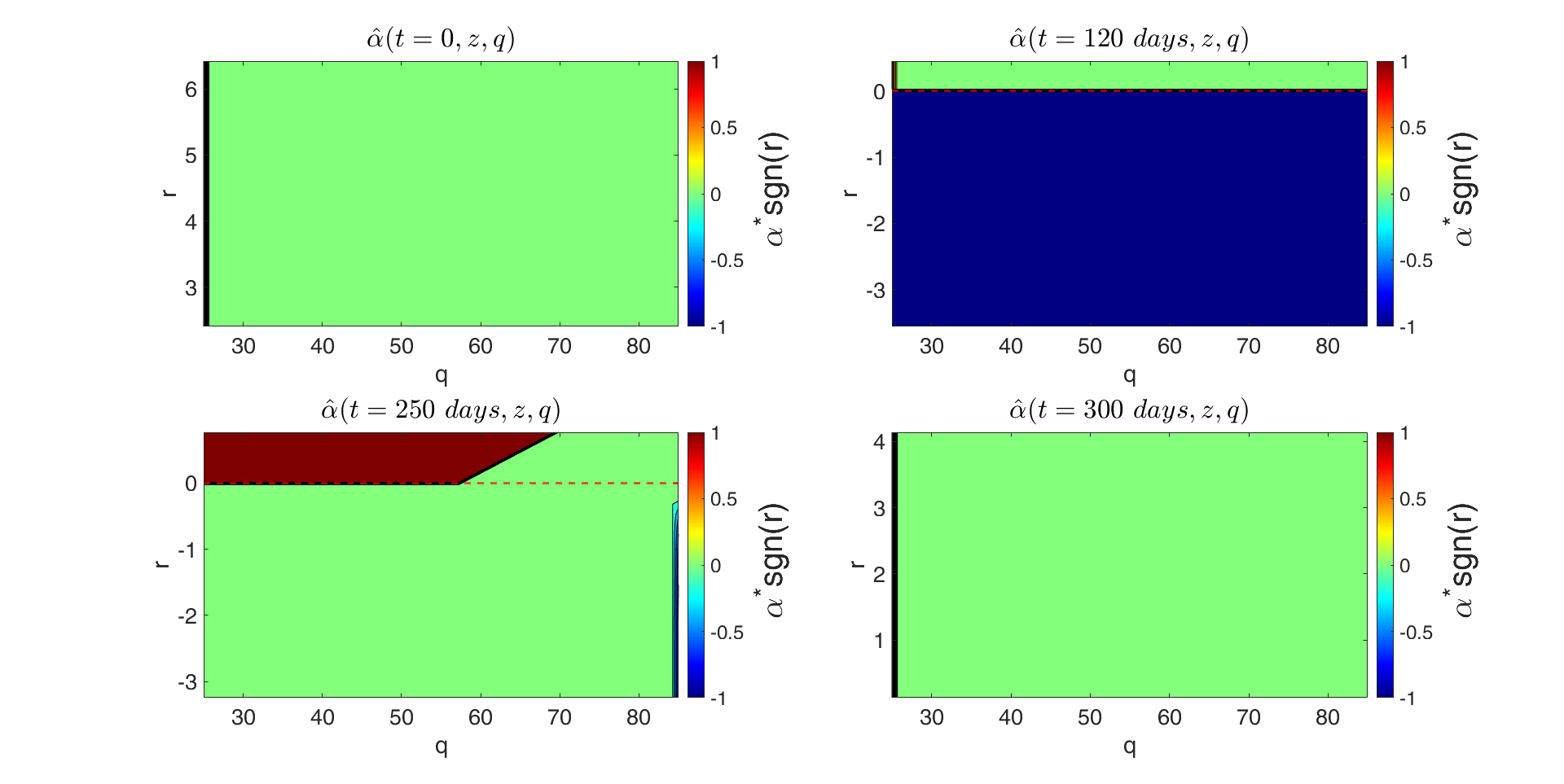}
				\caption[Optimal strategies at different time values for the strong seasonality case.]{Optimal strategy at $t = 0$ (Top left), $t = 120$ days (Top right), $t = 250$ days, (Bottom left), $t = 300$ days (Bottom right) for the strong seasonality case.}
				\label{NZG1Valfunct1StrongNew}
			\end{figure}
			
			In Figure \ref{NZG1Valfunct2StrongNew}, we also study the behavior of the value function and optimal strategy with respect to ES temperature $q$ and the residual demand $r$at time $t=230~$days. For the ES temperature level, we consider $q = 25,~55,~85^{\circ}\text{C}$ and $r = -4.40,~-1.67,~-0.39~kW$ for the residual demand. Since the top and bottom left panels follow a similar interpretation as in the basic case above, we focus solely on the top and bottom right panel. We observe in the top right panel that for $r = -4.40,~-1.67~kW$, the prosumer stores all excess production in the ES and only sells to the CHS for a full ES. For $r = -0.39~kW$, the prosumer also charges the ES and sells to the CHS if the ES is full or empty. It does not charge the ES immediately when it is empty since more overproduction is expecting in the future. In the bottom right panel, the prosumer charges the ES for $q = 25,~55^{\circ}\text{C}$ and compensates for the loss to the environment for $q = q_{max}$.
			
			\begin{figure}[ht!]			
				\hspace{-1.5cm}
				\includegraphics[width=1.15\textwidth]{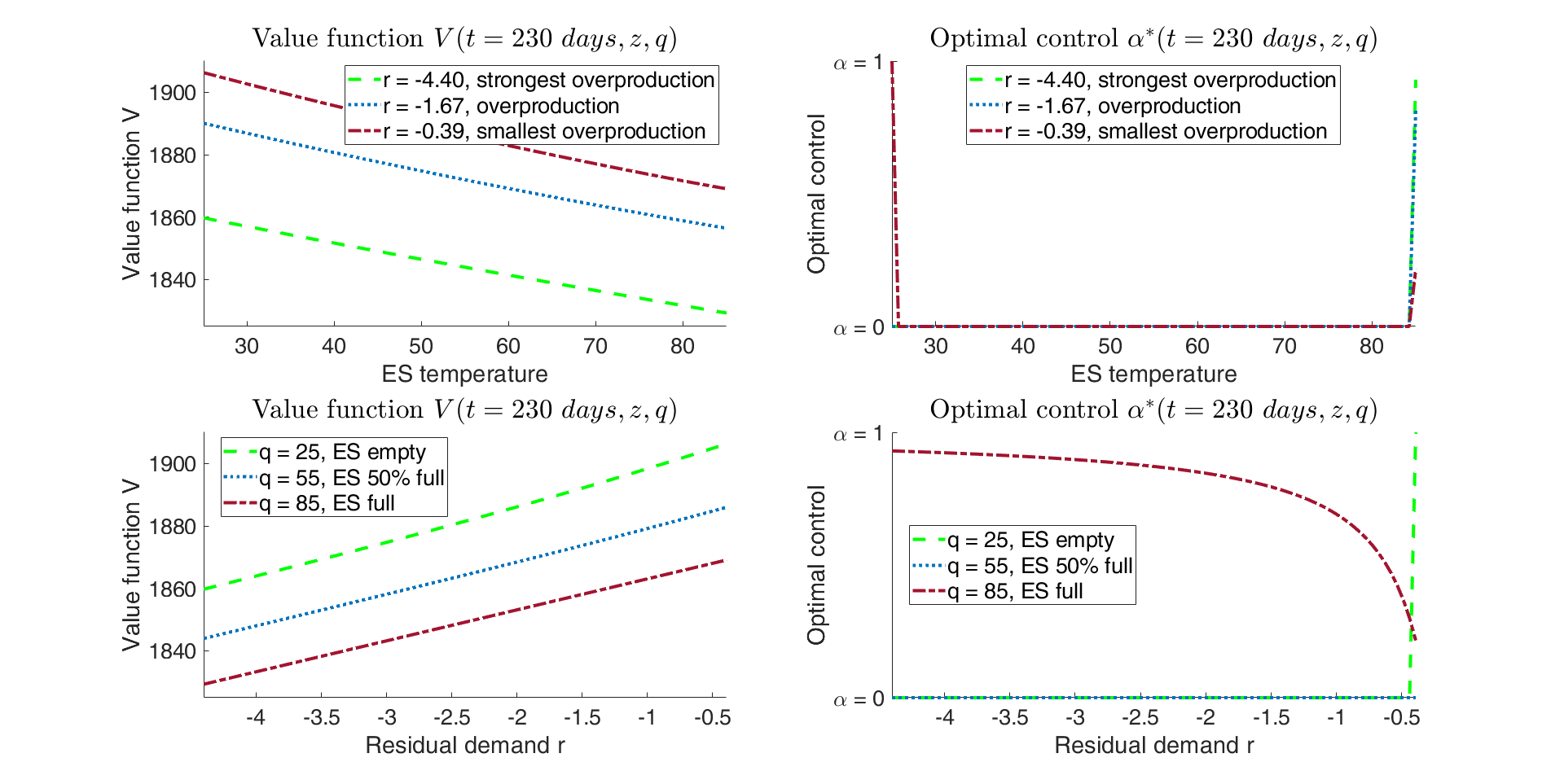}
				\caption[Value function and optimal strategy as functions of $Q^\alpha$ for different values of $R$ at $t=230~$days for the strong seasonality case.]{Value function (left) and optimal strategy (right) as functions of ES temperature $q$ and residual demand $r$ at $t=230~$days for the strong seasonality case.}
				\label{NZG1Valfunct2StrongNew}
			\end{figure}
			
			\paragraph{Case III: Value function and optimal strategy for a weak insulation}
			In this paragraph, we present and discuss the numerical results for the weak insulation case. For the seasonality function, we consider the same parameters $L^{Basic}_0,~L^{Basic}$ with values provided in Table \ref{tab:3}. Let $\gamma^{Weak}$ denote the heat transfer coefficient to the cold environment for a storage with weak insulation and take $\gamma^{Weak} = 2\gamma^{Basic} = 4.68\times 10^{-4}~\frac{kW}{m^2 K}$.
			In all panels of Figure \ref{NZG1Valfunct1WeakNew}, when faced with an unsatisfied demand, the prosumer discharges the ES for $q > q_{min}$ and otherwise purchases all thermal energy from the CHS for $q = q_{min}$. In the top left and down right panels, the prosumer sells most overproduction to the CHS, while in the top right panel, it sells all the overproduction to the CHS. In the down right panel, the prosumer first increases the ES level and afterward sells the overproduction to the CHS. 
			
			\begin{figure}[ht!]			
				\hspace{-1.25cm}
				\includegraphics[width=1.15\textwidth]{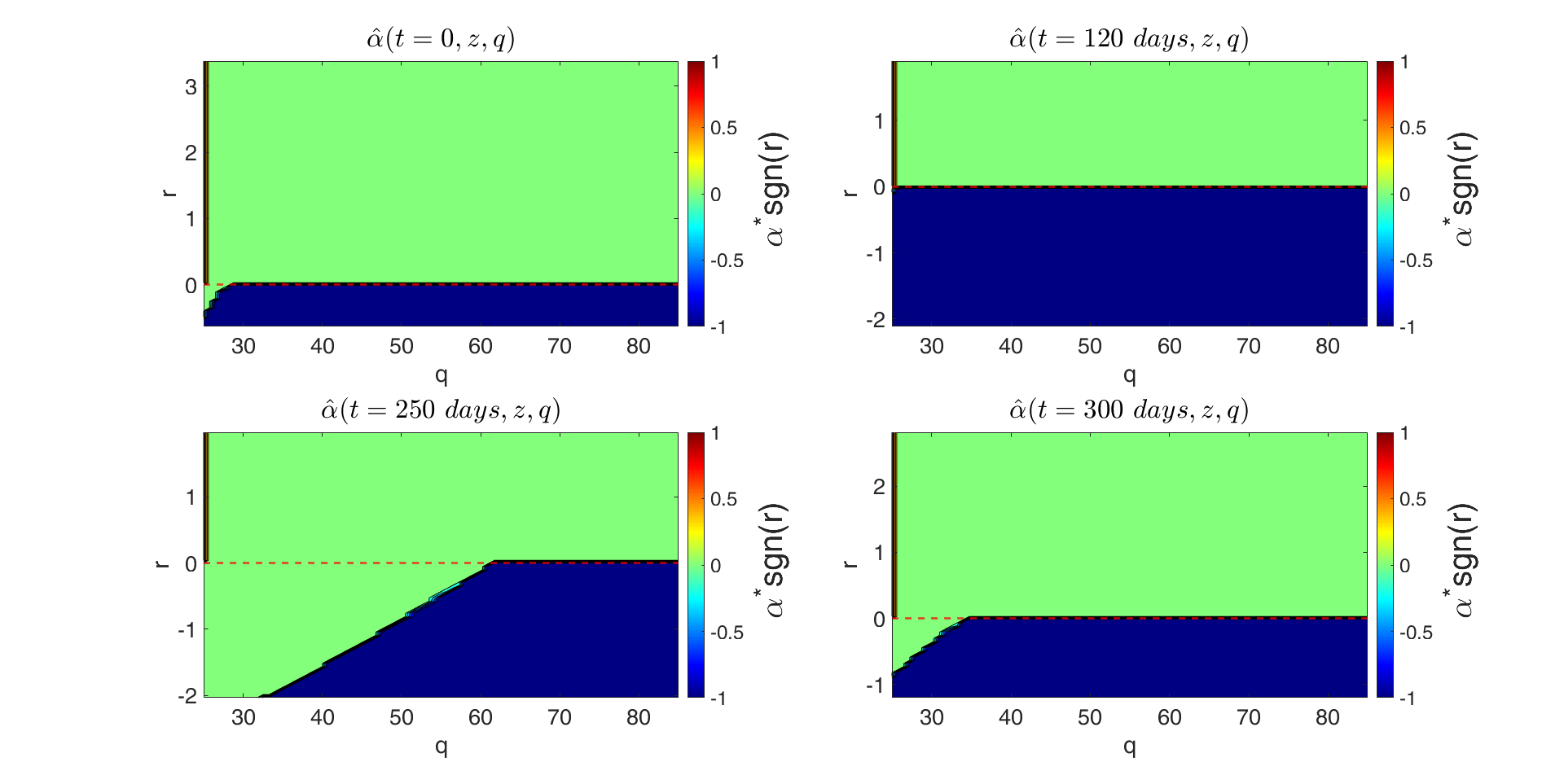}
				\caption[Optimal strategies at different time values for the weak insulation case.]{Optimal strategy at $t = 0$ (Top left), $t = 120$ days (Top right), $t = 250$ days, (Bottom left), $t = 300$ days (Bottom right) for the weak insulation case.}
				\label{NZG1Valfunct1WeakNew}
			\end{figure}
			
			In Figure \ref{NZG1Valfunct2WeakNew}, we focus the interpretaton on the top and bottom right panels since the other panels follow a similar interpretation as in the previous cases.\newline
			In the top right panel, in the case of an unsatisfied demand, the prosumer purchases all the thermal energy from the CHS if the ES is empty and otherwise discharges the ES if $q>q_{max}$. On the other hand, all the smallest unsatisfied demand is satisfied from discharging the ES. In the event of the strongest overproduction, the prosumer charges the ES while $q<q_{max}$ and subsequently sells the excess production to CHS only if $q=q_{max}$. \newline
			In the bottom right panel, for $q = 25^{\circ}\text{C},~55^{\circ}\text{C}$, the prosumer stores the excess production in the ES, whereas for an unsatisfied demand it discharges the ES. In the event of a full storage, the prosumer mostly compensates for the loss if faced with an overproduction. Furthermore, for an unsatisfied demand, the prosumer discharges the ES.
			
			\begin{figure}[ht!]			
				\hspace{-1.5cm}
				\includegraphics[width=1.15\textwidth]{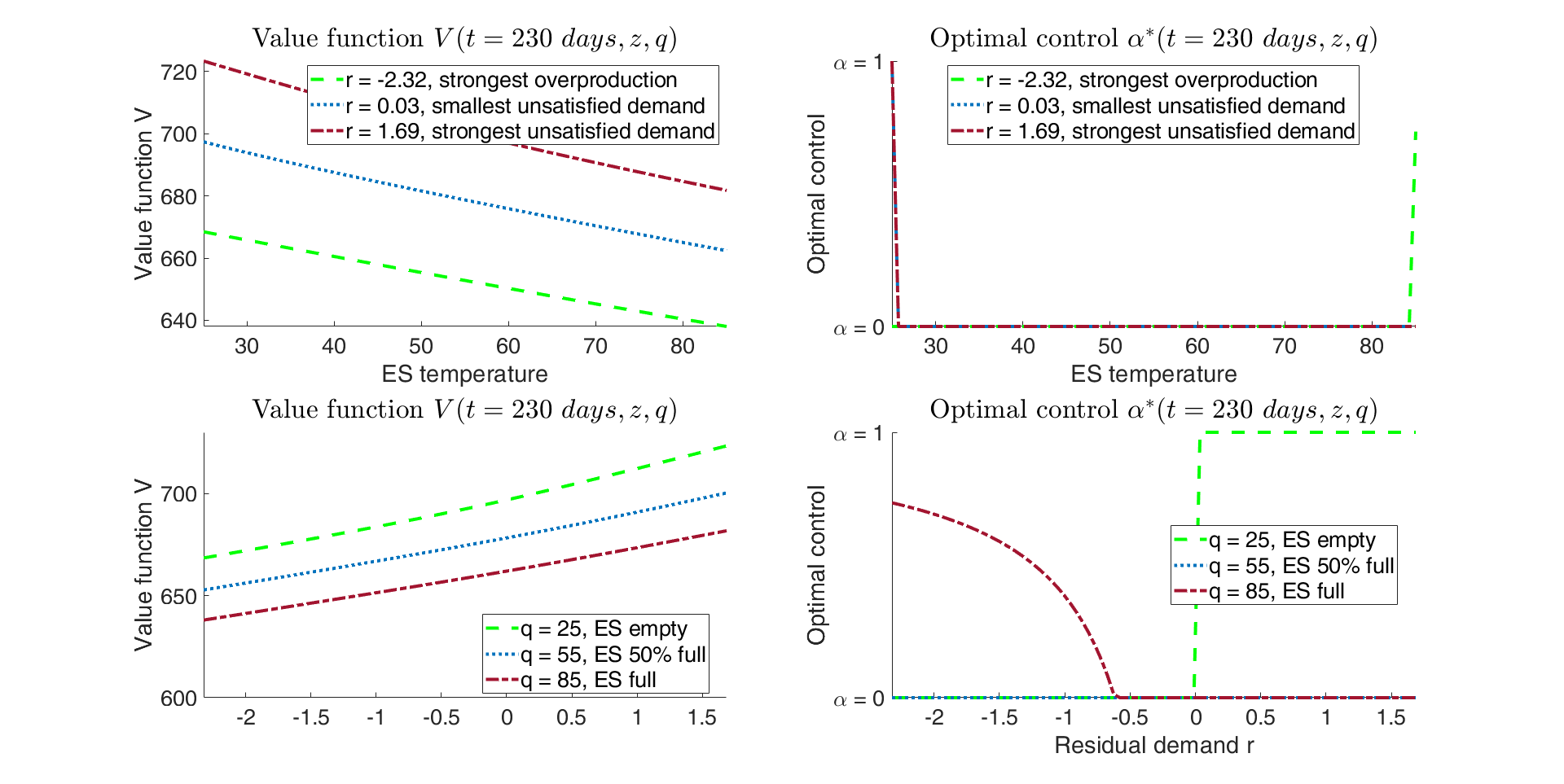}
				\caption[Value function and optimal strategy as functions of $Q^\alpha$ for different values of $R$ at $t=230~$days for the weak insulation case.]{Value function (left) and optimal strategy (right) as functions of ES temperature $q$ and residual demand $r$ at $t=230~$days for the weak insulation case.}
				\label{NZG1Valfunct2WeakNew}
			\end{figure}
			In Table \ref{tab:2} below, we present the maximum value function for the cases discussed above as well as the case of perfect insulation. Let $V^{Basic}_{max},~V^{Weak}_{max},~V^{Strong}_{max},~ V^{Perfect}_{max}$ denote the maximum value function for the basic, weak insulation, strong seasonality and perfect insulation cases, respectively. As expected, facing the same residual demand, a storage with a perfect insulation yields the smallest expected aggregated discounted cost of satisfying the hot water and heating demand of a household. In addition, we observe that in comparison to the basic case, a storage with weak insulation incurs a higher expected aggregated discounted cost. Furthermore for a strong seasonality, resulting in a bigger residual demand, the prosumer incurs a higher expected aggregated discounted cost as compared to the basic case.  
			\begin{table}[!h]
				\begin{center}
					\begin{tabular}{ |c|c||c|c| } 
						\hline
						Value function & Numerical values & Value function & Numerical values \\
						\hline 
						$V^{Perfect}_{max}$ & 1312.7 & $V^{Weak}_{max}$ & 1468.2\\ [5pt]
						$V^{Basic}_{max}$ & 1436.3 & $V^{Strong}_{max}$ & 3755.1\\ [.5pt]
						\hline
					\end{tabular}
					\caption[Maximum value function for all cases in Euros.]{Maximum value function for all cases in Euros.}
					\label{tab:2}
				\end{center}
			\end{table}
			\subsection{Optimal path of temperature level in ES}
			In the following, we present the optimal path $Q^{\alpha^*}$ for the respective optimal controls compute above. In all the figures, for $t\in [0,T]$, we show in addition to $Q^{\alpha^*}(t)$, the corresponding optimal absolute strategy, $\alpha^*(t)|R(t)|$, the seasonality function, $\muR(t)$, the residual demand, $R(t)$ as well as the minimim and maximum temperature levels $q_{min}$ and $q_{min}$. In all the figures below, we start with a fully charged external storage. We recall that Figures \ref{OptimalESBasic} and \ref{OptimalESWeak} have the same $R$ with different storage insulations. \newline
			In Figure \ref{OptimalESBasic}, we present $Q^{\alpha^*}$ for the basic case. For an unsatisfied demand, the prosumer discharges the ES until it is empty. While ES is empty, the prosumer satisfies all unsatisfied demand by purchasing thermal energy from the CHS. Until  mid May, most of the overproduction is sold to CHS for a revenue. From mid May, all overproduction is stored until the ES becomes full. We also note the prosumer compensates for the loss to keep the ES full during the overproduction. Once the ES is fully charged, most unsatisfied demands are satisfied by discharging the ES. Since the ES has a good insulation, the prosumer drives ES full earlier.   
			\begin{figure}[ht!]			
				\centering
				\includegraphics[width=1\textwidth]{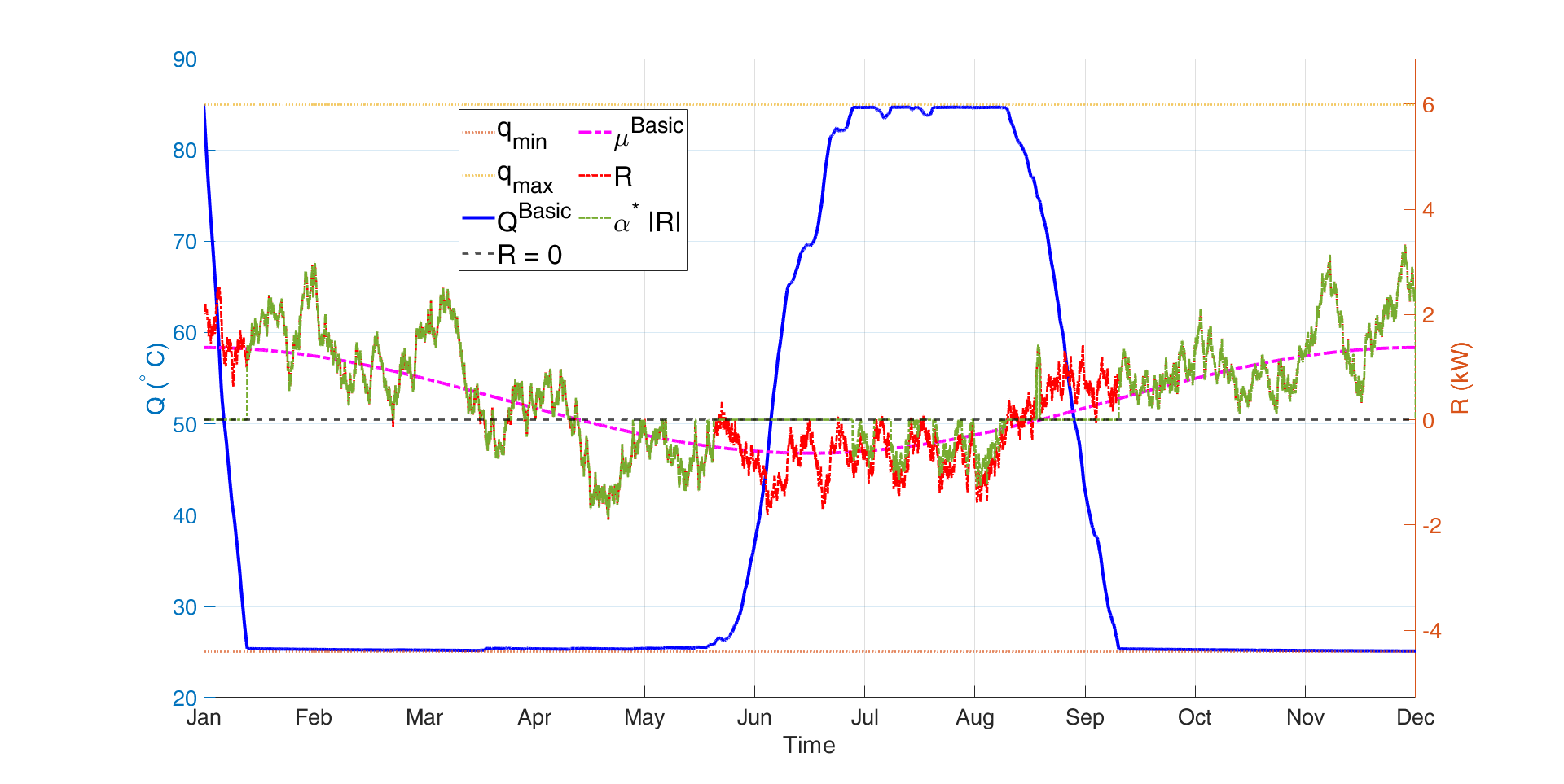}
				\caption[Optimal path of $Q^\alpha$ for the basic case and a liquidation problem.]{Optimal path of $Q^\alpha$ for the basic case and  a liquidation problem.}
				\label{OptimalESBasic}
			\end{figure}
			
			In Figure \ref{OptimalESWeak}, we present $Q^{\alpha^*}$ for the weak  case. Unlike in Figure \ref{OptimalESBasic}, we observe that the prosumer could not drive the ES full due to the weak insulation.
			\begin{figure}[ht!]			
				\centering
				\includegraphics[width=1\textwidth]{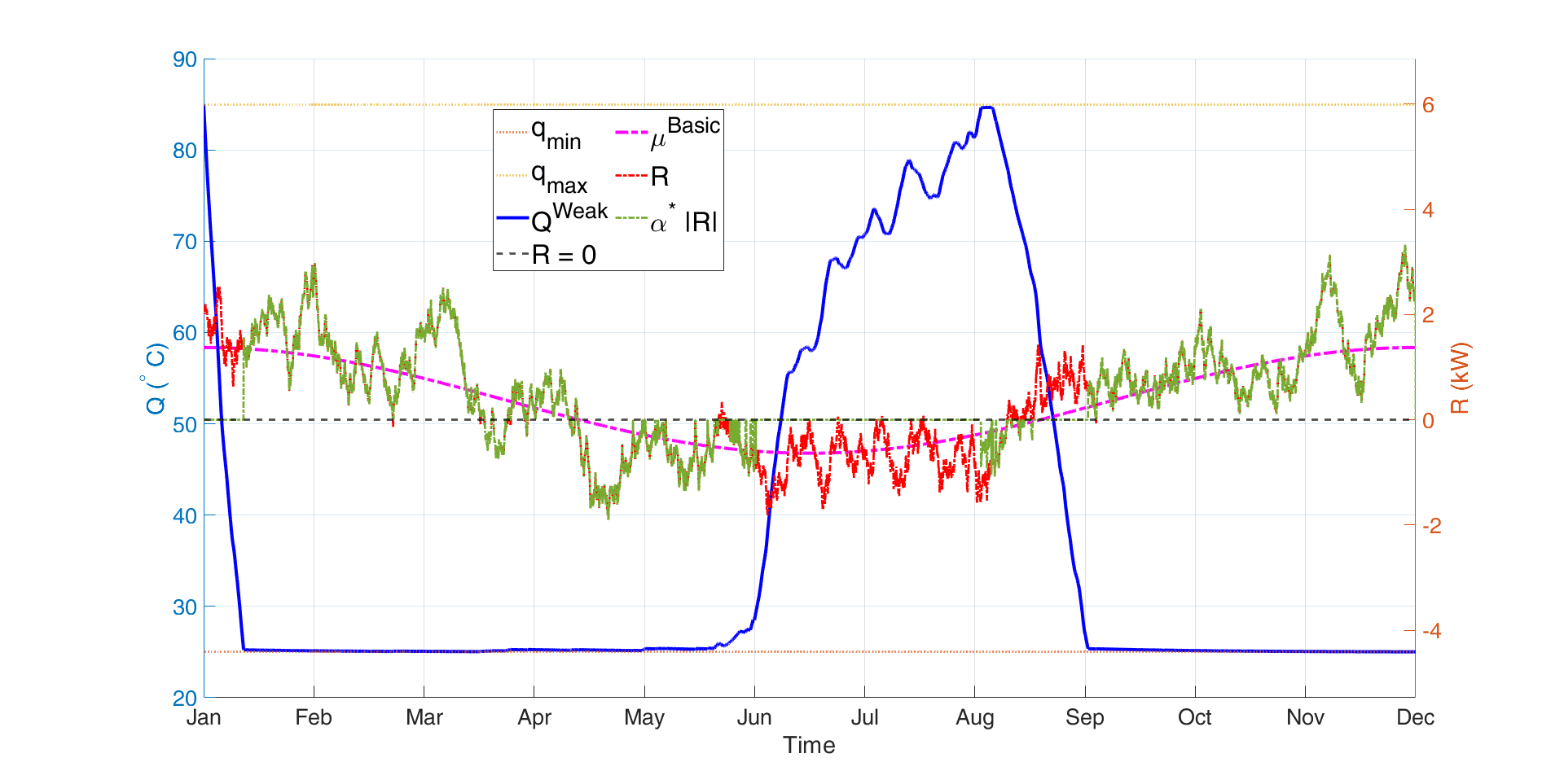}
				\caption[Optimal path of $Q^\alpha$ for the weak insulation case and a liquidation problem.]{Optimal path of $Q^\alpha$ for the weak insulation case and  a liquidation problem.}
				\label{OptimalESWeak}
			\end{figure}
			
			In Figure \ref{OptimalESStrong}, we present $Q^{\alpha^*}$ for the strong seasonality case. We recall that in this case, the prosumer is equipped with a storage with a good insulation. For an unsatisfied demand, the prosumer discharges the ES until it is completely empty. While the ES is an empty, all unsatisfied demands are satisfied from the CHS at a cost. Until June, all overproduction is sold to CHS for a revenue tough the ES is empty. From mid-June, the prosumer slowly charges the ES until it becomes completely full. Once the ES is full, the prosumer compensates the loss to the environment for any excess production. 
			\begin{figure}[ht!]			
				\centering
				\includegraphics[width=1\textwidth]{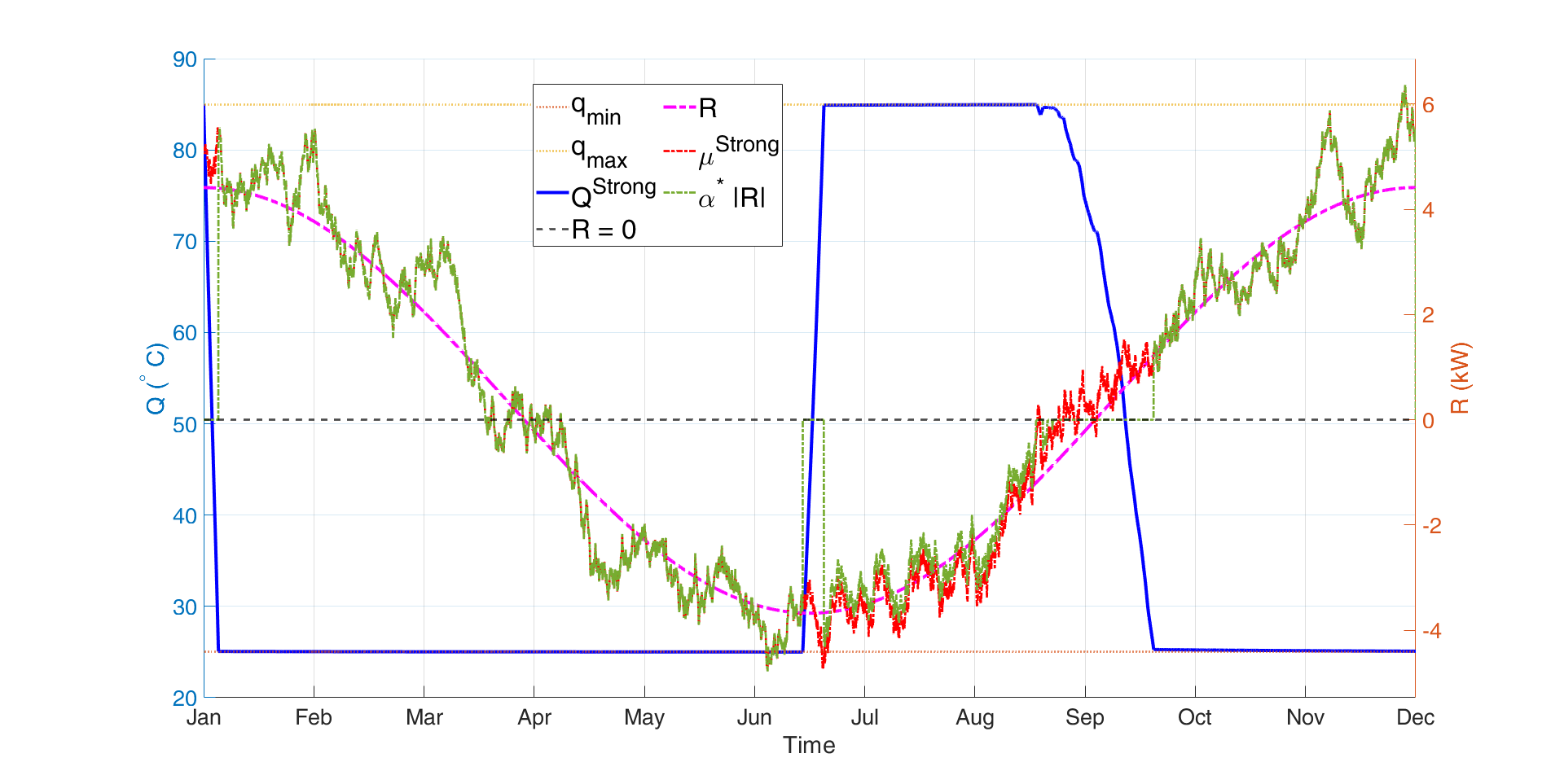}
				\caption[Optimal path of $Q^\alpha$ for the strong seasonality case and a liquidation problem.]{Optimal path of $Q^\alpha$ for the strong seasonality case and  a liquidation problem.}
				\label{OptimalESStrong}
			\end{figure}

			\section{Conclusion}
			\label{conclusion}
			In this project, we have investigated the optimal management of a residential heating system under uncertainty of the future residual demand. The aim to minimize the expected aggregated discounted cost of satisfying hot water and heat demand led to the formulation of a stochastic optimal problem under constraints which naturally arose from our modelling perspective. We derived the Hamilton-Jacobi-Bellmann (HJB) equation and motivated a discrete-time scheme to perform the numerical experiments and obtain the optimal strategies and value function.
			We illustred the results in 3 cases consisting of the basic setting, the weak insulation and the strong seasonality cases. We also considered the limiting case of perfect insulation and subsequently compared the maximum values functions (see Table \ref{tab:2}).
			
			\newpage
			\begin{table}[!h]
				\begin{center}
					\begin{tabular}{ |l r|c c| } 
						\hline
						Description &  & Numerical values & Units \\
						\hline
						\textbf{Residual demand}&  &  &   \\ [0pt]
						Speed of mean reversion & $\kappaR$ & 0.0063 & \text{h}$^{-1}$ \\ [5pt] 
						Volatility & $\sigmaR$ & 0.075 & \text{kW}/h$^{1/2}$ \\ [5pt]
						Basic case long term residual demand & $L^{Basic}_0$ & $0.37$ & \text{kW} \\ [5pt]
						Basic case amplitude of seasonality & $L^{Basic}$ & $1.00$ & \text{kW} \\ [5pt]
						Strong seasonality case long term residual demand & $L^{Strong}_0$ & $1.38$ & \text{kW} \\ [5pt]
						Strong seasonality case amplitude of seasonality & $L^{Strong}$ & $3.04$ & \text{kW} \\ [5pt]
						Length of seasonality & $\rho$ & 8760 & h \\ [5pt]
						Reference time of seasonality & $\widetilde{t}$ & 0 & h \\ [5pt]
						\hline
						\textbf{External storages}&  &  &   \\ [.5pt]
						Calibration time for $\gamma^{Basic}$ & $t^*$ & 720 & h \\ [5pt]
						Mass of water & $m_Q$ & 7854 & kg \\ [5pt]  
						Specific heat capacity of water & $c_P$ & $0.0012$ & kWh/kg K  \\ [5pt]
						Radius & $r_d$ & 1 & $m$\\ [5pt]
						Height & $h_0$ & 2.5 & $m$\\ [5pt]
						Surface area & $A$ & 21.99 & $m^2$\\ [5pt] 
						Volume & $V_0$ & 7.854 & $m^3$\\ [5pt]
						Energy capacity & $\mathcal{E}$ & 547.94 & \text{kWh} \\
						Heat transfer coefficient for basic case & $\gamma^{Basic}$ & $2.34\times 10^{-4}$ &  \text{kW/$m^2$ K}\\ [5pt]
						Heat transfer coefficient for weak insulation case & $\gamma^{Weak}$ & $4.68\times 10^{-4}$ &  \text{kW/$m^2$ K}\\ [5pt]
						Perfect heat transfer coefficient & $\gamma^{Perfect}$ & $0$ &  \text{kW/$m^2$ K}\\ [5pt] 
						Minimum temperature & $q_{min}$ & 25 & $^{\circ}$C \\ [5pt] 
						Maximum temperature & $q_{max}$ & 85 & $^{\circ}$C \\ [5pt]
						Calibration temperature for $\gamma$ & $\widetilde{q}$ & $65$ &  $^{\circ}$C \\
						\hline
						\textbf{Cost functions}&  &  &   \\ [.5pt]
						Discounting factor & $\delta $ & $1.712\times10^-6$ & h$^{-1}$  \\ [5pt]
						Electricity price & $S_{el}$ & 0.33 & \EUR/kWh \\ [5pt]
						Liquidation price & $S_{liq}$ & 0.004 & \EUR/kWh\\ [5pt]
						Penalty price & $S_{pen}$ & 0.32 & \EUR/kWh\\ [5pt]
						Ordinary pump penalty constant & $d_1$ & $0.01$ &  \\ [5pt]
						Heat pump penalty constant & $d_2$ & $0.012$ & \text{K}$^{-1}$ \\ [5pt]
						Heat pump output temperature & $\pi_d$ & 25 & $^{\circ}$C \\  [5pt]
						Pipe constant temperature & $P_c$ & 20 & $^{\circ}$C  \\
						Penalty temperature in ES & $q_{pen}$ & 50 & $^{\circ}$C \\ [5pt] 
						Time horizon & $T$ & 8760 & h \\ 
						\hline
						\textbf{Discretization}&  &  &   \\ [.5pt]
						Time step & $\Delta t$ & 1 & \text{h} \\ 
						Number of time points & $N_t$ & 8760 &  \\ 
						Number of grid points for q & $N_q$ & $80$ &  \\
						Number of grid points for $\zR$ & $N_{\zR}$ & $85$ &  \\
						\hline
					\end{tabular}
					\caption[Model and discretization parameters.]{Model and discretization parameters.}
					\label{tab:1}
				\end{center}
			\end{table}
			
			\newpage \clearpage
			\begin{appendix}	
				\section{List of Notations}
				\label{append_a}
				\begin{longtable}{p{0.3\textwidth}p{0.68\textwidth}l}
					$Q,~\ZR,~R$ & Temperature in ES, deseasonalized and residual demand &\\
					$T$ & Finite time horizon &\\
					$\alpha$ & Heat buying or selling rate & \\
					$\muR$ & Yearly seasonality function &\\
					$\kappaR,~\sigmaR$ & Mean reversion speed and volatility of residual demand &\\
					$L^{Basic}_0,~L^{Basic}$ & Basic case long term mean residual demand and seasonality function amplitude &\\
					$L^{Strong}_0,~L^{Strong}$ & Strong seasonality case long term mean residual demand and seasonality function amplitude &\\
					$\rho$ & Length of seasonality &\\
					$\tilde{t}$ & Reference time &\\
					$m_Q$ & Mass of water in the external storage &\\
					$c_P$ & Water specific heat capacity &\\
					$A$ & Surface area of external storage &\\
					$\gamma^{Basic}$ & Basic case heat transfer coefficient &\\
					$\gamma^{Weak}$ & Weak insulation case heat transfer coefficient &\\
					$Q_{amb}$ & Ambient temperature &\\
					$L_0^S,~L^S$ & Long term mean and amplitude of seasonality of heat price &\\
					$S_{buy},~S_{sell}$ & Heat buying and selling price &\\
					$S_{el}$ & Electricity price & \\
					$S_{liq},~S_{pen}$ & Liquidation and penalty prices & \\
					$\xi$ & Spread in heat price formulation &\\
					$q_{min},~q_{max}$ & Minimum and maximum storage levels &\\
					$\mathcal{K}_a$, $\mathcal{K}_d$ & Continuous, discrete-time state dependent set of feasible controls &\\
					$J,~V$ & Performance criterion and value function at time $t$ &\\ 
					$\Psi,~\Phi$ & Running cashflow and terminal cost & \\
					$\delta$ & Discount rate & \\
					$C_0$ & Cost of buying from or revenue of selling to CHS &\\
					$C_1,~C_2$ & Cost of using the heat and ordinary pumps & \\
					$P_c$ & Constant temperature in connecting pipe & \\
					$\pi_{min},~\pi_d$ & Minimum and desired temperatures in internal storage & \\
					$\tilde{a}$ & Decision rule & \\
					$\mathcal{A}$ & Set of admissible controls & \\
					$\mathcal{X}$ & Computational domain & \\
					$\overline{\mathcal{X}}$ & Truncated computational domain & \\
					$\mathcal{S}$ & Truncated domain of residual demand & \\
					$\epsilon$ & Tolerance level & \\
					$\zR_{min},~\zR_{max}$ & Minimum and maximum residual demands $\ZR$ & \\
					$\muR_{min},~\muR_{max}$ & Minimum and maximum seasonality function & \\
					$s_0$ & Asymptotic standard deviation of residual demand & \\
					$\mathcal{G} = \mathcal{G}_t\times \mathcal{G}_{N_{\zR}}\times \mathcal{G}_q$ & Discretized time, residual demand and temperature level & \\
					$t_n$ & Discrete time points & \\
					$\Delta t,~\Delta \zR,~\Delta q$ & time, residual demand and temperature level step sizes & \\
					$N_t,~N_{\zR},~N_q$ & Number of time, residual demand and temperature level points & \\
					$\Q^a$ & Path of the temperature in external storage for a fix control $a$ & \\
					$\Q^{n+1}_{j(\ell,n)}$ & Approximation of $\mathcal{Q}^a$ at discrete time $n+1$ & \\
					$\underline{\chi},~\overline{\chi}$ & Minimum buying and maximum selling rates &\\
					$\mathcal{L}_{\ZR},~\mathcal{L}_Q$ & Differential operator of residual demand and temperature level &\\
					$\mathcal{L}$ & Differential operator of shifted $\mathcal{L}_{\ZR}$ &\\
					$\Lambda$ & Discretized differential operator &\\
					$V^n_{\ell,j}$ & Value function at discrete time $n$ and grid point $(\ell,j)$ &\\
					$V^{n+1}_{j(\ell,n)}$ & Interpolation of $V^n_{\ell,j}$ &\\
					$\mathcal{E}$ & Maximum amount of energy stored in ES &\\
				\end{longtable}

				\section{Calibration of parameters}
				In this section we discuss the calibration of some model parameters and let $Q_{amb}(t) = q_{min}$ for simplicity. 
				
				\subsection{Calibration of $\gamma^{Basic}$}
				\label{calibration_gamma}
				We consider a full storage at initial time $t_0 = 0$ and assume no storage decision (i.e. $\alpha = 1$). We seek to determine the heat transfer coefficient $\gamma$ such that at a prescribed time $t^*$, the temperature in ES is at a prescribed level $\widetilde{q}$. Thus we consider
				\begin{align}
					\label{cal1}
					\mathrm{d}Q(t) &= -\frac{A\gamma}{m_Qc_{P}}(Q(t)-q_{min})\mathrm{d}t, \quad Q(t_0)=q_{max},
				\end{align} 
				requiring $Q(t^*) = \widetilde{q}$. Solving Equation \eqref{cal1} using the integrating factor technique, we deduce
				\begin{equation}
					\label{gamma}
					\gamma = \frac{m_Qc_P}{At^*}\log\Big(\frac{q_{max}-q_{min}}{\widetilde{q}-q_{min}}\Big).
				\end{equation}
				Taking $t^*=30~$days in Equation \eqref{gamma} yields $\gamma^{Basic}$.
				
				\subsection{Calibration of $L^{Basic}_0$, $L^{Basic}$, $L^{Strong}_0$, and $L^{Strong}$}
				\label{calibrationL_0L}
				For simplicity, we consider one seasonality component, i.e. $m=1$. Let $L_0+L$ denote the worst-case (strongest) residual demand and consider at time $t_0 = 0$ a fully charged ES. We assume that this fully charged ES can satisfy the demand $L_0+L$ until becoming empty at $t_s$. We now consider
				\begin{align}
					\label{cal2}
					\mathrm{d}Q(t) &= -\frac{1}{m_Qc_{P}}(L_0+L + A\gamma (Q(t)- q_{min}))\mathrm{d}t, \quad Q(t_0)=q_{max},
				\end{align}
				and require $Q(t_s) = q_{min}$.\newline
				First we suppose there is no seasonality (i.e. $L=0$) and take $t_s = t_1$. Solving Equation \eqref{cal2}, we deduce
				\begin{equation}
					\label{L0}
					L_0 = \frac{A\gamma(q_{max}- q_{min})}{e^{A_1}-1}, \quad \text{where} \quad  A_1 = \frac{A\gamma}{m_Qc_P}t_1.
				\end{equation}
				Similarly in Equation \eqref{cal2}, if $L\ne 0$ and taking $t_s = t_2$, we obtain using the integrating factor that
				\begin{equation}
					\label{L}
					L = \frac{A\gamma(q_{max}-q_{min})}{e^{A_2}-1} - L_0, \quad \text{with} \quad A_2 = \frac{A\gamma}{m_Qc_P}t_2.
				\end{equation}
				From Equations \eqref{L0} and \eqref{L}, taking $t_1 = 45~$days, $t_2 = 15~$days and $\gamma = \gamma^{Basic}$ give $L^{Basic}_0$ and $L^{Basic}$, respectively. Furthermore, taking $L^{Strong}_0 = L^{Basic}_0$, $t_2 = 5~$days and $\gamma = \gamma^{Basic}$ in Equation \eqref{L}, yield $L^{Strong}$. 
				
				\subsection{Calibration of $d_1$ and $d_2$}
				\label{calibrationd1_d2}
				Let $\underline{S}_{sell} := \displaystyle \min_{t\in [0,T]} S_{sell}(t)$ and $\underline{S}_{buy} := \displaystyle \min_{t\in [0,T]} S_{buy}(t)$ denote the minimum heat selling and buying prices, respectively. \newline
				From our modelling perspective, we expect the revenue from selling overproduction to CHS to be greater than the cost of electricity consumption to move the hot water. That is 
				\begin{equation}
					\label{d1}
					d_1S_{el} < \underline{S}_{sell}  \qquad \text{implying} \qquad d_1 < \frac{\underline{S}_{sell}}{S_{el}}.
				\end{equation}
				Hence we choose $d_1$ satisfying condition \eqref{d1}. \newline
				
				Similarly, we calibrate $d_2$ such that the cost of using the heat pump is smaller than the cost of buying heat from the CHS. That is 
				\begin{equation}
					\label{d2}
					(d_1+d_2(\pi_d-P_c))S_{el} < \underline{S}_{buy} \qquad \text{implying} \qquad d_2 < \frac{\underline{S}_{buy}-d_1S_{el}}{S_{el}(\pi_d-P_c)}.
				\end{equation}
				Similarly, we choose $d_2$ such that \eqref{d2} is satisfied.
				
				\section{Proofs}
				
				\subsection{Proof of Theorem \ref{derivDS}}
				\label{proofderivDS}
				Fix $(t_n, \zR_\ell, q_j)$ and let $\Delta t = t_{n+1}-t_n$ and rewrite Equation \eqref{dpp} as
				\begin{equation}
					\label{dDPE}
					V(t_n,\zR_\ell,q_j) = \inf_{\alpha \in \mathcal{A}}\mathbb{E}_{t_n,\zR_\ell,q_j}\Big[\int_{t_n}^{t_{n+1}}e^{-\delta (t-t_n)}\Psi(t,\ZR(t),\alpha(t))\mathrm{d}t+e^{-\delta \Delta t}V(t_{n+1},\ZR(t_{n+1}),Q^{\alpha}(t_{n+1}))\Big].
				\end{equation}
				Suppose that for $t\in [t_n,t_{n+1})$, $\ZR(t) = \zR_\ell$ (fixed and known) and $a = \alpha(t) = \alpha^n_{\ell, j}$ (unknown but constant).	Thus equation (\ref{dDPE}) becomes
				\begin{align}
					\label{dDPE1}
					V(t_n,\zR_\ell,q_j) &\approx \inf_{a \in \mathcal{K}^n_d(\zR_\ell,q_j)}\Big\{\int_{t_n}^{t_{n+1}}e^{-\delta (t-t_n)}\Psi(t,\zR_\ell,a)\mathrm{d}t+e^{-\delta \Delta t}V(t_{n+1},\zR_\ell,Q^a(t_{n+1}))\Big\}, 
				\end{align}
				where $V(t_{n+1},\cdot,\cdot)$ is known and $V(t_n,\cdot,\cdot)$ is to be determined.
				Substituting (\ref{dSLT}) into (\ref{dDPE1}) we derive the optimal control
				\begin{equation}
					\alpha^{*n}_{\ell, j} = \argmin_{a\in \mathcal{K}_d(\zR_\ell,q_j)}\Big\{\int_{t_n}^{t_{n+1}}e^{-\delta (t-t_n)}\Psi(t,\zR_\ell,a)\mathrm{d}t+e^{-\delta \Delta t}V(t_{n+1},\zR_\ell,\Q^{a,n+1}_{j(\ell,n)})\Big\}.
				\end{equation}
				Now we evaluate \eqref{dDPE} on the path $\Q^{\alpha^*}$ for the constant control $\alpha^* = \alpha^{*n}_{\ell, j}$ and ``release" the fixed residual demand and consider it as a stochastic process to obtain
				\begin{equation}
					\label{dDPE2}
					V(t_n,\zR_\ell,q_j) \approx \mathbb{E}_{t_n,\zR_\ell,q_j}\Big[\int_{t_n}^{t_{n+1}}e^{-\delta (t-t_n)}\Psi(\ZR(t),\alpha^{*n}_{\ell, j})\mathrm{d}t+e^{-\delta \Delta t}V(t_{n+1},\ZR(t_{n+1}),\Q^{\alpha^*,n+1}_{j(\ell,n)})\Big].
				\end{equation}
				From \eqref{dDPE2}, set $U(t,\zR) = V(t,\zR,\Q^{\alpha^*})$ and apply Feynman-Kac's formula to obtain the PDE
				\begin{align}
					\label{dHJB1}
					\frac{\partial U}{\partial t}(t,\zR) + \mathcal{L}U(t,\zR) + \Psi(r,\alpha^*) &= 0, \qquad \text{on~} [t_n,t_{n+1})\times \overline{\mathcal{X}}, \\
					U(t_{n+1},\zR) &= V(t_{n+1},\zR,\Q^{\alpha^*,n+1}_{j(\ell,n)}), 
				\end{align}
				where $\mathcal{L}U=\mathcal{L}_{\ZR}U-\delta U$.
				
				\subsection{Proof of Proposition \ref{positcond}}
				\label{proofpositcond}
				Let $\ell = 1,\ldots,N_{\zR}-1$ ~\text{and}~ $n = 0,\ldots,N_t-1$. From Equations \eqref{positive_theta} and \eqref{negative_theta} we note that if $\mathcal{D}_\ell \ge 0$ and $\mathcal{H}_\ell \ge 0$, then $\mathcal{F}_\ell \ge 0$. Hence, it is enough to investigate the positivity of $\mathcal{D}_\ell$ and $\mathcal{H}_\ell$.\newline
				From \eqref{positive_theta}, we have $\mathcal{H}_\ell\ge 0$. Now $\mathcal{D}_\ell\ge 0$ requires the condition $\sigmaR^2 \ge -2\Delta \zR \kappaR \zR_\ell$. \newline
				Taking the maximum on both sides of the above equation and recalling the expression of $\zR_{min}$ derived from the $3\sigma$-rule, yields $\sigmaR^2 \ge -2\Delta \zR \kappaR \zR_{min} = 6\Delta \zR \kappaR s_0$. \newline
				Also from Equation \eqref{negative_theta}, we note that $\mathcal{D}_\ell\ge 0$. $\mathcal{H}_\ell\ge 0$ requires $\sigmaR^2 \ge 2\Delta \zR \kappaR \zR_\ell$. \newline
				Again, taking the maximum on both sides  and recalling the expression of $\zR_{max}$ obtained from the $3\sigma$-rule, we derive $\sigmaR^2 \ge 2\Delta \zR \kappaR \zR_{max} = 6\Delta \zR \kappaR s_0$.\newline
				
				Therefore, from the above discussions, we derive the positivity condition, $\sigmaR^2 \ge 6\Delta \zR \kappaR s_0$.
				
				\subsection{Proof of Proposition \ref{lininterpol}}
				\label{prooflininterpol}
				Assume the CFL condition holds and let $b^n_{\ell,j}=\frac{\Delta t}{m_Qc_P\Delta q}\Big[(1-a)(\muR_n+ \zR_\ell) + A\gamma(q_j-Q_{amb}^n)\Big]$. Then, we wish to interpolate $V(t_{n+1},\zR_\ell,\Q^{a,n+1}_{j(\ell,n)})$ by function values $V^{n+1}_{\ell,j}$ and $V^{n+1}_{\ell,j-1}$.\newline
				If $\Q^{a,n+1}_{j(\ell,n)}$ lies between $q_j-\Delta q$ and $q_j$, then $V^{n+1}_{j(\ell,n)}$ is represented by
				\begin{equation}
					\label{interpolate1}
					V^{n+1}_{j(\ell,n)} = \Big(1-b^n_{\ell,j}\Big)V^{n+1}_{\ell,j}+b^n_{\ell,j}V^{n+1}_{\ell,j-1} \quad b^n_{\ell,j}\ge 0,\quad j=1,\ldots,N_q.
				\end{equation}
				Similarly, if $\Q^{a,n+1}_{j(\ell,n)}$ lies between $q_j$ and $q_j+\Delta q$, the linear interpolation is given by
				\begin{equation}
					\label{interpolate2}
					V^{n+1}_{j(\ell,n)} = \Big(1+b^n_{\ell,j}\Big)V^{n+1}_{\ell,j}-b^n_{\ell,j}V^{n+1}_{\ell,j+1} \quad b^n_{\ell,j}< 0,\quad j=0,\ldots,N_q-1.
				\end{equation}
				Combining both equations \ref{interpolate1} and \ref{interpolate2}, we derive
				\begin{equation}
					\label{Eq}
					V^{n+1}_{j(\ell,n)} = \left\{ \begin{array}{ll}
						\Big(1-b^n_{\ell,j}\Big)V^{n+1}_{\ell,j}+b^n_{\ell,j}V^{n+1}_{\ell,j-1} & \quad b^n_{\ell,j}\ge 0,\quad j=1,\ldots,N_q, \\
						\Big(1+b^n_{\ell,j}\Big)V^{n+1}_{\ell,j}-b^n_{\ell,j}V^{n+1}_{\ell,j+1} & \quad b^n_{\ell,j}< 0,\quad j=0,\ldots,N_q-1. 
					\end{array} \right.
				\end{equation}	
				From Equation \eqref{Eq} we obtain the following equivalent representation
				\begin{align}
					\label{approx}
					V^{n+1}_{j(\ell,n)} &= V^{n+1}_{\ell,j}-\frac{b^n_{\ell,j}+|b^n_{\ell,j}|}{2}\Big(V^{n+1}_{\ell,j}-V^{n+1}_{\ell,j-1}\Big)-\frac{b^n_{\ell,j}-|b^n_{\ell,j}|}{2}\Big(V^{n+1}_{\ell,j+1}-V^{n+1}_{\ell,j}\Big), \nonumber\\
					&= \Big(1-|b^n_{\ell,j}|\Big)V^{n+1}_{\ell,j}+\frac{b^n_{\ell,j}+|b^n_{\ell,j}|}{2}V^{n+1}_{\ell,j-1}-\frac{b^n_{\ell,j}-|b^n_{\ell,j}|}{2}V^{n+1}_{\ell,j+1}, \nonumber\\
					&= B^{(q,n)}_{\ell,j}V^{n+1}_{\ell,j}+A^{(q,n)}_{\ell,j}V^{n+1}_{\ell,j-1}+C^{(q,n)}_{\ell,j}V^{n+1}_{\ell,j+1},
				\end{align} 
				for $\ell=0,\ldots,N_{\zR},~j=1,\ldots,N_{q}-1,~n=0,\ldots,N_t-1$ and where $B^{(q,n)}_{\ell,j} = 1-|b^n_{\ell,j}|$, $A^{(q,n)}_{\ell,j} = \frac{b^n_{\ell,j}+|b^n_{\ell,j}|}{2}$ and $C^{(q,n)}_{\ell,j} = -\frac{b^n_{\ell,j}-|b^n_{\ell,j}|}{2}$.

				\section{Matrix $\mathcal{B}$ and $\mathcal{C}$}
				\label{MatrixBC}       
				Let $\mathcal{I}_{N_{\zR}-1}$ denote the $(N_{\zR}-1)\times (N_{\zR}-1)$ identity matrix. Then, the  $(N_{\zR}-1)\times (N_{\zR}-1)$ diagonal matrices $\mathcal{B}$ in Equation \eqref{linsys2} and $\mathcal{C}$ in Equation \eqref{linsys3} are given respectively by            
				
				\begin{equation}
					\mathcal{B} = (1+\Delta t \mathcal{F}_0)\mathcal{I}_{N_{\zR}-1}, \quad \mathcal{B} = (1+\Delta t \mathcal{F}_{N_{\zR}})\mathcal{I}_{N_{\zR}-1}.
				\end{equation}
				
			\end{appendix}	
			
			\begin{acknowledgements}
				The authors thank Olivier Menoukeu Pamen and Paul Honor\'e Takam for helpful discussions.
			\end{acknowledgements}
			
			
			\bibliographystyle{acm}
			\bibliography{references}
			\addcontentsline{toc}{chapter}{References}
		\end{document}